\documentclass[11pt]{article}

\newcommand{\F}{\mathbb F}
\newcommand{\C}{\mathcal C}

\newcommand{\OO}{\mathcal O}

\usepackage{lmodern}
\usepackage[utf8]{inputenc}
\usepackage{amsmath}
\usepackage{amsfonts}
\usepackage{amssymb}
\usepackage{amsmath}
\usepackage{amsthm}
\usepackage[a4paper,left=3.3cm,right=3.3cm,bottom=2.2cm,top=2cm,foot=1.3cm,includeheadfoot=true]{geometry}
\usepackage{framed}
\usepackage[boxed]{algorithm2e}
\usepackage[hidelinks]{hyperref}
\usepackage{multirow,bigdelim}
\usepackage[babel=true]{csquotes} 
\usepackage{tikz}
\usepackage{url}

\newcommand{\Og}{O_g}
\DeclareMathOperator{\gothp}{\mathfrak{p}}

\DeclareMathOperator{\Q}{\mathbb{Q}}
\DeclareMathOperator{\Z}{\mathbb{Z}}

\DeclareMathOperator{\ord}{\mathrm{ord}}
\newcommand{\softO}{\widetilde{O}}
\DeclareMathOperator{\Jac}{\mathrm{Jac}}

\DeclareMathOperator{\Res}{\mathrm{Res}}

\DeclareMathOperator{\Disc}{\mathrm{Disc}}

\DeclareMathOperator{\End}{\mathrm{End}}

\DeclareMathOperator{\Gal}{\mathrm{Gal}}

\DeclareMathOperator{\LC}{\mathrm{LC}}

\DeclareMathOperator{\Frac}{\mathrm{Frac}}

\DeclareMathOperator{\GCD}{GCD}
\newcommand{\R}{\mathbb{R}}

\newtheorem{lemma}{Lemma}
\newtheorem{proposition}[lemma]{Proposition}
\newtheorem{definition}[lemma]{Definition}
\newtheorem{theorem}[lemma]{Theorem}

\DeclareMathOperator{\Ker}{Ker}
\SetKwInOut{Input}{input}
\SetKwInOut{Output}{output}

\newcounter{countproblem}

\begin{document}
\title{Counting points on hyperelliptic curves with explicit real multiplication in arbitrary genus}
\author{Simon Abelard \\[0.2cm]
  {\normalsize Université de Lorraine, CNRS, Inria} \\
   {\normalsize University of Waterloo }}
   \date{    simon.abelard@uwaterloo.ca \\
   Cheriton School of Computer Science\\
   Waterloo, Ontario, N2L 3G1 \\
   Canada
}

\maketitle

\begin{abstract} We present a probabilistic Las Vegas algorithm for computing
  the local zeta function of a genus-$g$ hyperelliptic curve defined over
  $\mathbb F_q$ with explicit real multiplication (RM) by an order $\Z[\eta]$
  in a degree-$g$ totally real number field.

  It is based on the approaches by Schoof and Pila in a more favorable case
  where we can split the $\ell$-torsion into $g$ kernels of endomorphisms, as
  introduced by Gaudry, Kohel, and Smith in genus 2. To deal with these kernels
  in any genus, we adapt a technique that the author, Gaudry, and
  Spaenlehauer introduced to model the $\ell$-torsion by structured polynomial
  systems.  Applying this technique to the kernels, the systems we obtain are
  much smaller and so is the complexity of solving them.

  Our main result is that there exists a constant $c>0$ such that, for any
  fixed $g$, this algorithm has expected time and space complexity $O((\log
  q)^{c})$ as $q$ grows and the characteristic is large enough. We prove that
  $c\le 9$ and we also conjecture that the result still holds for $c=7$.



\end{abstract}

\section{Introduction}

Due to its numerous applications in cryptology, number theory, algebraic
geometry or even as a primitive used in other algorithms, the problem of
counting points on curves and Abelian varieties has been extensively studied
over the past three decades. Among the milestones in the history of
point-counting, one can mention the first polynomial-time algorithm by
Schoof~\cite{Sc85} for counting points on elliptic curves, and the subsequent
extension to Abelian varieties by Pila~\cite{Pi90}. Using similar approaches,
we design a probabilistic algorithm for computing the local zeta functions of
hyperelliptic curves of arbitrary fixed genus $g$ with explicit real
multiplication and bound its complexity.

Given an Abelian variety of dimension $g$ over a finite field $\F_{q}$, Pila's
algorithm computes its local zeta function in time $(\log q)^\Delta$, where
$\Delta$ is doubly exponential in $g$. Further contributions were made
in~\cite{HI98,AH01} so that this exponent $\Delta$ is now proven to be polynomial
in $g$ in general, and even linear in the hyperelliptic case~\cite{AGS17}.

In genus $2$, a tailor-made extension of Schoof's algorithm due to Gaudry,
Harley and Schost~\cite{GH00,GS04,GS12} allows to count points in time
$\softO(\log^8q)$. Yet, this remains much larger than the complexity of the
Schoof-Elkies-Atkin (SEA) algorithm~\cite{Sc95}, which is the standard for
elliptic point-counting in large characteristic and runs in $\softO(\log^4q)$
bit operations. For genus-2 curves with explicit real multiplication (RM), i.e.
curves having an additional endomorphism for which an explicit expression is
known, a much more efficient point-counting algorithm is introduced
in~\cite{gaudry2011counting} with a bit complexity in $\softO(\log^5q)$, thus
narrowing the gap between genus 1 and 2. 

These algorithms were extended to genus-3 hyperelliptic curves in~\cite{AGS18}
with an asymptotic complexity in $\softO(\log^{14}q)$ bit operations that is
decreased to $\softO(\log^{6}q)$ bit operations when the curve has explicit RM.

The aim of this paper is to study the asymptotic complexity of point-counting
on hyperelliptic curves with explicit RM when $g$ is arbitrary large. In this
case, we bound the exponent of $\log q$ by $9$ and therefore remove the
dependency on $g$ from the exponent of $\log q$.

Another way to avoid such a painful dependency in $g$ in the complexity without
restricting to such particular cases is to use the $p$-adic methods, in the
spirit of Satoh's and Kedlaya's algorithms~\cite{Satoh00,Ked01} for elliptic
and hyperelliptic curves. These methods have also been extended beyond the
hyperelliptic case~\cite{tuitman2017counting,costa} and one can also mention
the algorithms of Lauder and Lauder-Wan that also hold for very general
varieties~\cite{Lauder04,LaWa02}.  Although these methods are polynomial in
$g$, they are exponential in $\log p$ and therefore cannot be used in large
characteristic. 

Indeed, the $\ell$-adic approaches derived from Schoof's algorithm and the
$p$-adic approaches are complementary when either $g$ or $p$ is small but we
still lack a classical algorithm running in time polynomial in both $g$ and
$\log q$. However, for counting points on reductions modulo many primes $p$ of
the same curve, an algorithm introduced by Harvey in~\cite{Harvey14} is
polynomial in $g$ and polynomial \emph{on average} in $\log p$.

In this paper, we follow the spirit of the Schoof-Pila algorithm and recover
the local zeta function by computing the characteristic polynomial $\chi_\pi$
of the action of the Frobenius endomorphism $\pi$ on the $\ell$-torsion
subgroups for sufficiently many primes $\ell$. The key to our complexity
result is that, thanks to the real multiplication, it is sufficient to have
$\pi$ act on much smaller subgroups of the $\ell$-torsion, at least for a
postive proportion of the primes $\ell$. The following definition sums up the
assumptions that we make on our particular (families of) curves. 

\begin{definition}[Explicit real multiplication]\label{def:rmexp}
  We say that a curve $\C$ has explicit real
multiplication by $\mathbb Z[\eta]$ if the subring $\mathbb Z[\eta]\subset
\End(\Jac(\C))$ is isomorphic to an order in a totally real degree-$g$ number
field, and if we have explicit formulas describing $\eta(P-P_\infty)$ for some
fixed base point $P_\infty$ and a generic point $P$ of $C$. 
\end{definition}

\paragraph*{Remark.}

Once a rational Weierstrass point $P_\infty$ is picked on $\C$, we represent
elements (reduced divisors) of $\Jac\C$ as formal sums $\sum_{i=1}^w
(P_i-P_\infty)$ and call $w$ the weight of the divisor. Alternatively, we
represent elements of $\Jac\C$ using the Mumford form $\langle u,v\rangle$
where $u$ and $v$ are polynomials in $\F_q[X]$ with $\deg u=w$ and $u| v^2-f$.
We refer to~\cite[Sec.~4.4 \& 14.1]{handbook} for more background on Jacobians
of hyperelliptic curves. In cases where $\C$ does not have an
odd-degree Weierstrass model, we can work in an extension of degree at most
$2g+2$ of the base field in order to ensure the existence of a rational
Weierstrass point. 

 By explicit formulas,
we mean $2g+2$ polynomials in $\mathbb F_q[x,y]$ which we denote by $(\eta^{(u)}_i(x,y))_{i\in\{0,1,\ldots,g\}}$ and
$(\eta^{(v)}_i(x,y))_{i\in\{0,1,\ldots,g\}}$ such that,
when $\C$ is given in odd-degree Weierstrass form, the Mumford coordinates of
$\eta( (x, y)-P_\infty)$ are $$\Big\langle X^g+ \sum_{i=0}^{g-1}\big({\eta^{(u)}_i(x,y)}/{\eta^{(u)}_g(x,y)}\big) X^i,
\sum_{i=0}^{g-1}
\big({\eta^{(v)}_i(x,y)}/{\eta^{(v)}_g(x,y)}\big)X^i\Big\rangle,$$ where $(x,y)$
is the generic point of the curve.
\smallskip

As in~\cite{gaudry2011counting,AGS18}, we consider primes $\ell\in\mathbb Z$
such that $\ell\mathbb Z[\eta]$ splits as a product $\gothp_1\cdots\gothp_g$ of
prime ideals. Computing the kernels of an endomorphism $\alpha_i$ in each
$\gothp_i$ provides us with an algebraic representation of the $\ell$-torsion
$\Jac(\C)[\ell]\subset \Ker \alpha_1 + \cdots + \Ker\alpha_g$. Then, we compute
from this representation integers $a_0,\ldots,a_{g-1}$ in $\Z/\ell\Z$ such that
the sum $\pi+\pi^{\vee}$ of the Frobenius endomorphism and its dual equals
$a_0+a_1\eta+\cdots +a_{g-1}\eta^{g-1}\bmod \ell$. Once enough modular
information is known, the values of the $a_i$'s such that $\pi+\pi^{\vee} =
\sum_{i=0}^{g-1}a_i\eta_i$ are recovered via the Chinese Remainder Theorem and
the coefficients of the characteristic polynomial of the Frobenius can be
directly expressed in terms of the $a_i$'s.
 
Computing the kernels of the endomorphisms $\alpha_i$ is the dominant step in
terms of complexity and thus the cornerstone of our result. We still model
these kernels by polynomial systems that we then have to solve, but the
resultant-based techniques that were used in~\cite{gaudry2011counting}
and~\cite{AGS18} are no longer satisfying when $g$ is arbitrary large. We
therefore use the modelling strategy of~\cite{AGS17} and apply it to the
kernels instead of applying it to the whole $\ell$-torsion. The polynomial
systems we derive from this approach are in fact very similar to those
of~\cite{AGS17}, except that our kernels are comparable in size to the
\enquote{$\ell^{1/g}$-torsion}, resulting in much smaller degrees and
ultimately in a complexity gain by a factor $g$ in the exponent of $\log q$,
decreasing it from linear to constant.  Using the geometric resolution
algorithm just as in~\cite{AGS17}, we solve these systems in time $K(\log
q)^{9+o(1)}$ where $K$ depends on $\eta$ (and thus on $g$ too) but not on $q$.
It is interesting to note that this result suffers from the pessimistic cubic
bounds on the degrees of Cantor's polynomials established in~\cite{AGS17} and
that---assuming quadratic bounds as proven in genus 1, 2 and 3---we get a
complexity in $K(\log q)^{7+o(1)}$, which is close to the complexity bound
proven in~\cite{AGS18} for genus-3 hyperelliptic curves with explicit RM. 

For hyperelliptic curves with RM, we have thus been able to eliminate the
dependency in $g$ in the exponent of $\log q$, but this does not mean that our
algorithm reaches polynomial-time complexity in both $g$ and $\log q$. Indeed,
we also discuss the reasons why the \enquote{constant} $K$ depends
exponentially on $g$. Among them, we shall see that some can actually be
discarded by considering even more particular cases while some appear to be
inherent to our geometric-resolution based approach. This remaining exponential
dependency also explains why this algorithm is currently not a practical one in
genus $\ge 4$, although its complexity seems close to that of the algorithm
presented in~\cite{AGS18}.

\paragraph{Organization.} 

In Section~\ref{sec:overanyg}, we give an overview of our point-counting
algorithm, along with an example of families of hyperelliptic curves of
arbitrary high genus with RM by a real subfield of a cyclotomic field. In
particular, we prove a bound on the size and number of primes $\ell$ to consider
in our algorithm. Section~\ref{sec:cplxrmanyg} focuses on the main primitive of
our algorithm: the computation of a non-zero element in the kernel of an
endomorphism $\alpha$ whose degree is a small multiple of $\ell^2$. This section
adapts methods and results of~\cite[Sec.~4 \& 5]{AGS17} to design structured
polynomial systems whose solution sets are subsets of $J[\alpha]$.
Section~\ref{sec:cplxrmangyg} concludes on the complexity of solving these
systems, and on the overall complexity result. We also
present an analysis on the dependency of the final complexity in $g$,
investigating the various places where exponential factors may occur and how to
avoid them when it is possible.

\section{Overview}\label{sec:overanyg} 
The main result of this paper can be summarized by the following theorem, which makes the dependency on $\eta$ explicit. 
\begin{theorem}\label{th:mainanyg}
  For any $g$ and any $\eta\in\overline{\Q}$ such that $\Q(\eta)$ is a
  totally-real number field of degree $g$, there exists an explicitly
  computable $c(\eta)>0$ such that there is an integer $q_0(g,\eta)$ such that
  for all prime power $q=p^n$ larger than $q_0(g,\eta)$ with $p\ge (\log
  q)^{c(\eta)}$ and for all genus-$g$ hyperelliptic curves $\C$ with explicit RM
  by $\Z[\eta]$ defined over $\F_q$, the local zeta function of $\C$ can be
  computed with a probabilistic algorithm in expected time bounded by $(\log
  q)^{c(\eta)}$.
\end{theorem}

In Section~\ref{sec:cplxrmangyg}, we also bound $c$ by $9+\varepsilon(\eta)$
and conjecture that it should be $7+\varepsilon(\eta)$, where the term
$\varepsilon(\eta)$ is added to take into account factors depending only on
$\eta$ and factors in $O((\log q)^{\varepsilon})$ for any $\varepsilon >0$.  In
the whole paper, we will make use of the following notation to implicitly
include those terms.

\paragraph*{Notation.}
We say that a function $F(\eta,q)$ is in $O_\eta((\log q)^c)$ if for any fixed
$\eta$ and any $\varepsilon >0$ we have $F(\eta,q)\in O((\log
q)^{c+\varepsilon})$, in the sense of the usual $O()$-notation. Using this
notation, our point-counting algorithm runs in time
$O_\eta((\log q)^9)$.  We emphasize once more that this notation includes
polylogarithmic factors in $q$.

\paragraph*{Remark.}
It is also possible to explicit the lower bound on $p$ required by Theorem~\ref{th:mainanyg}. Indeed, $p$ is only constrained by Proposition~\ref{prop:prop3} of Section~\ref{sec:cplxsolsys}. With the $O_\eta()$ notation, this condition amounts to $p$ being greater than a quantity in $O_\eta( (\log p)^3)$.

\subsection{Families of RM curves}\label{sec:rmfamily}

We present one-dimensional families of hyperelliptic curves from~\cite{TTV},
constructed via cyclotomic covers. They have an affine model
$\C_{n,t}:Y^2=D_n(X)+t$, where $t$ is a parameter and $D_n$ is the $n$-th
Dickson polynomial with parameter 1 defined inductively by $D_0(X)=2$,
$D_1(X)=X$, and \[D_n(X)=XD_{n-1}(X)-D_{n-2}(X).\] Since $D_n(X)$ has degree
$n$, setting $n=2g+1$ for odd $n$ yields a one-dimensional family $\C_{n,t}$ of
genus $g$ hyperelliptic curves given by an odd-degree Weierstrass model. Their
Jacobians all have an explicit endomorphism $\eta$, and when $n$ is prime,
~\cite[Prop.~2]{KoSm} shows that $\Z[\eta]\cong\Z[\zeta_n+\zeta_n^{-1}]$,
where $\zeta_n$ is a primitive $n$-th root of unity over $\Q$. Note that the
construction of an explicit endomorphism is still possible whenever $n=2g+1$ is
not prime, but then the curves in $\C_{n,t}$ have non-simple Jacobians, which
means there are better alternatives than using our algorithm for counting points
on them.

Another family based on Artin-Schreier covering is detailed in the same paper
but these curves have genus $(p-1)/2$ where $p$ is the characteristic of the
base field, so that our complexity study using the $O_\eta()$ notation would be
pointless in that case.  Since $g$ becomes much larger than $\log p$ in that
case, it would be more efficient to use $p$-adic algorithms anyway.

Let $\C$ be a (genus-$g$) hyperelliptic curve in the family $\C_{2g+1,t}$,
defined over a finite field $\F_q$.  In~\cite{KoSm}, Kohel and Smith compute
formulas for the Mumford form of $\eta\left( (x,y)-P_{\infty} \right)$, where
$(x,y)$ is the generic point on $\C$. These formulas are given explicitly for
some examples in genus 2 and 3, and an algorithm~\cite[Algorithm~5]{KoSm} is
presented to compute them for any $\C$.  This algorithm has a time complexity
in $O(g^2)$ field operations and requires to store $O(g^3)$ field elements.
Thus, given a curve from that family as input, an explicit endomorphism of its
Jacobian can be computed once and for all in $\softO(g^3\log q)$ time and space
complexity, which is negligible compared to the cost of counting points on the
curve.  

\subsection{The characteristic equation}\label{sec:psipi}

Let us consider a genus-$g$ hyperelliptic curve $\C$ over $\F_q$ with explicit
RM in the sense of Definition~\ref{def:rmexp} and let $J$ be its Jacobian. We
denote by $\pi$ the Frobenius map $x\mapsto x^q$ over $\overline{\F_q}$. It extends
to an endomorphism of $J$ which we also denote by $\pi$. The dual endomorphism
of $\pi$ is denoted by $\pi^\vee$ and satisfies $\pi\pi^\vee=\pi^\vee\pi=q$.

Counting points (or computing local zeta functions) amounts to computing
$\chi_\pi$, the characteristic polynomial of the Frobenius endomorphism.
Following Schoof and Pila, we do it by computing $\chi_\pi\bmod\ell$ for
sufficiently many primes $\ell$ coprime to $q$, using the fact that
$\chi_\pi\bmod\ell$ is the characteristic polynomial of the Frobenius
endomorphism acting on $J[\ell]$.  This approach works thanks to the Weil
conjectures, which were proven by Dwork, Deligne and Grothendieck (see for
instance~\cite{Deligne} for the original proof). For our purpose, the important
consequences of the Weil conjectures are that $\chi_\pi=\sum_{i=0}^{2g} c_iX^i$
is a degree-$2g$ polynomial whose coefficients $c_i$ are integers such that for
any $i\le 2g$, we have $c_{i}=q^{g-i}c_{2g-i}$ and $\vert c_i\vert \le
{2g\choose i} q^{(2g-i)/2}$.

As in~\cite{gaudry2011counting,AGS18}, let us consider $\psi=\pi+\pi^{\vee}$
and recall that $\psi\in\Q[\eta]$. We still have $\psi\pi=\pi^2+q$ and once
again, we test this equation to determine $\psi$ instead of the characteristic
equation of $\pi$.  The link between $\psi$ and $\pi$ needs to be made
explicit, which is the aim of the present section.  Using the above relation
between the $c_i$'s, we can write
$$\chi_{\pi}(X)=\sum_{i=0}^g\sigma_i(X^{2g-i}+q^{g-i}X^{i}),$$ with
$\sigma_i=c_{2g-i}$ for $i\ne g$ and $\sigma_g=c_g/2$. Then we propagate the
Weil bounds to the $\sigma_i$'s and get $|\sigma_i|=|c_{2g-i}|\le {2g\choose i}
q^{i/2}$ for i$\ne g$ and $|\sigma_g|\le {2g\choose g} q^{g/2}/2$.

Since we have $q^{-g}(\pi^{\vee})^g\chi_{\pi}(\pi)=0$ by the Cayley-Hamilton
theorem, and using the fact that $\pi\pi^{\vee}=q$, we rewrite that as
\[\sum_{i=0}^g \sigma_{g-i}(\pi^i+(\pi^{\vee})^i)=0.\] 

Our plan is to compute $\chi_\pi\bmod\ell$ by determining $\psi$. Let us write
$\psi=\sum_{i=0}^{g-1}a_i\eta^i$, the goal of the section is to prove bounds on
the coefficients $a_i$, so that we can estimate the number and maximal size of
primes $\ell$ required to compute $\psi$ without ambiguity. Note that $\psi$ is
in the maximal order of $\Q(\eta)$, but not necessarily in $\Z[\eta]$.
However, as in~\cite{gaudry2011counting,AGS18}, $\Z[\eta]$ has finite index in
the maximal order and the possible common denominator of the $a_i$'s has to
divide $\left[\mathcal{O}_{\Q(\eta)}:\Z[\eta]\right]$.  This denominator
entails that additional primes may be required to fully determine $\psi$,
however $\left[\mathcal{O}_{\Q(\eta)}:\Z[\eta]\right]$ depends only on $\eta$
so that it will disappear in the $O_\eta$-notation of our complexity estimates.
Therefore, we do not detail further this subtlety and assume for simplicity
that the $a_i$'s are integers, which we wish to bound by $O_\eta(\sqrt{q})$. 

Let us first express the quantities $\pi^i+(\pi^{\vee})^i$ in terms of powers of
$\psi$ as a first step towards expressing the $\sigma_i$'s as functions of the
$a_i$'s.

\begin{lemma}\label{lem:psipi}
  For any $i\in\{1,\ldots,g\}$, there exist integers
  $(\alpha_{i,j})_{0\le j< i}$ such that $\alpha_{i,j}$ is in $O(q^{(i-j)/2})$ and
  \[\pi^i+(\pi^{\vee})^i=\psi^i+\sum_{j=0}^{i-1}\alpha_{i,j}\psi^j.\]
\end{lemma}
\begin{proof}
  The statement holds for $i=1$ with $\alpha_{1,0}=0$ by the definition of $\psi$.
  For $i=2$, we have $\psi^2=\pi^2+(\pi^{\vee})^2+2\pi\pi^{\vee}$, so that we
  have the result with $\alpha_{2,0}=-2q$ and $\alpha_{2,1}=0$. 
  
  In this proof, we set the convention $\alpha_{i,i}=1$ to simplify our
  recurrence relations. Let us now assume the lemma holds for any positive
  integer no greater than a certain $i$. We therefore have
\[\psi^{i+1}=(\pi+\pi^{\vee})\psi^i=(\pi+\pi^{\vee})\left[(\pi^i+(\pi^{\vee})^i)
-\sum_{j=0}^{i-1}\alpha_{i,j}\psi^j\right].\]
The first term is equal to
$\pi^{i+1}+(\pi^{\vee})^{i+1}+q(\pi^{i-1}+(\pi^{\vee})^{i-1})$ so that we can
use the lemma once again for $i-1$ and get
\[\psi^{i+1}=\pi^{i+1}+(\pi^{\vee})^{i+1}-\alpha_{i,i-1}\psi^i+q\alpha_{i-1,0}
+\sum_{j=1}^{i-1}(q\alpha_{i-1,j}-\alpha_{i,j-1})\psi^j.\]
Thus, we have computed the $\alpha_{i+1,j}$ given by
\[\alpha_{i+1,j}=\left\lbrace\begin{aligned}
     \alpha_{i,i-1}\qquad \quad\text{if $j=i$},&\\
     -q\alpha_{i-1,0}\qquad\quad\text{if $j=0$},& \\
      \alpha_{i,j-1}-q\alpha_{i-1,j} \quad \text{else.} &
\end{aligned}\right.\]
Let us now study the order of magnitude of the $\alpha_{i+1,j}$: from the
recurrence hypothesis on both $i$ and $i-1$, $\alpha_{i,i-1}=\alpha_{i+1,i}$ is in
$O(\sqrt{q})$, $\alpha_{i-1,0}$ is in $O(q^{(i-1)/2})$ so that
$\alpha_{i+1,0}$ is in $O(q^{(i+1)/2})$, and both $q\alpha_{i-1,j}$ and
$\alpha_{i,j-1}$ are in $O(q^{(i+1-j)/2})$,
which proves the result for any other $\alpha_{i+1,j}$. By induction, the lemma
is proven.
\end{proof}

Note that our $O$-notation in the previous statement and proof can be a bit
misleading as there may not be an absolute constant bounding all the
$\alpha_{i,j}/q^{(i-j)/2}$. However, from the recurrence relation between the
$a_{i,j}$'s, one sees that each $\alpha_{i,j}$ is equal to $q^{(i-j)/2}$ plus an
error term that is in $O_\eta(q^{(i-j-1)/2})$ and at worst quadratic in $g$, hence
the error term is negligible compared to $q^{(i-j)/2}$.

\begin{proposition}\label{prop:chipsi}
 The polynomial $\chi_\pi$ is uniquely determined by the coefficients $a_i$'s
 of $\psi$ in the basis $(1,\eta,\ldots,\eta^{g-1})$, and there exists
 $C_\eta>0$ depending only on $g$ and $\eta$ such that for any
 $i\in\{0,\ldots,g-1\}$, we have $\vert a_i\vert\le C_\eta\sqrt{q}$.
\end{proposition}
\begin{proof}
Recall that $\sigma_i$ is the coefficient $c_{2g-i}$ of $\chi_\pi$ for $i\ne g$, and that $\sigma_g=c_g/2$. Using Lemma~\ref{lem:psipi} for any $i\in\{1,\ldots,g\}$
and setting $\alpha_{i,i}=1$, we have
\[\sum_{i=0}^g\sigma_{g-i}\sum_{j=0}^{i}\alpha_{i,j}\psi^j=
\sum_{j=0}^g\psi^j\sum_{i=j}^g\alpha_{i,j}\sigma_{g-i}=0.\] Let us
define $s_j=\sum_{i=j}^g\alpha_{i,j}\sigma_{g-i}$ and
$\chi_{\psi}(X)=X^g+s_{g-1}X^{g-1}+\cdots+s_0$. We previously showed that for any $i$ the coefficient $\sigma_i$ is in $O(q^{i/2})$, therefore each $\sigma_{g-i}$ is in $O(q^{(g-i)/2})$ and by Lemma~\ref{lem:psipi} we know that $\alpha_{i,j}$ is in $O(q^{(i-j)/2})$, so  that each $s_j$ is in $O(q^{(g-j)/2})$.  

Note that $\chi_{\psi}$ is a degree-$g$ monic polynomial vanishing on $\psi$,
and it is therefore its characteristic polynomial. Let us denote by $\psi_k$
the $g$ conjugates of $\psi$, which are the $g$ real roots of $\chi_\psi$. By
the Fujiwara bounds from~\cite{fujiwara1916}, for any $k\in\{1,\ldots,g\}$ we
have \[ |\psi_k| \le 2\max_{0\le k \le g} \left( \vert
s_{g-k}\vert^{1/k}\right).\] We already know that $\vert s_{g-k}\vert=
O(\sqrt{q}^k)$, so we deduce that the $\vert\psi_k\vert$ are in $O(\sqrt{q})$.

The conjugates $\psi_k$ can be expressed explictly: by definition we
have $\psi=\sum_{i=0}^{g-1} a_i\eta^i$, so we can write
$\psi_k=\sum_{i=0}^{g-1} a_i\eta_k^i$ by choosing a convenient order for the
$g$ conjugates of $\eta$ (possibly in the Galois-closure of $\Q(\eta)$). We
can view this operation as a change of variables from $a_i$'s to $\psi_k$'s and
the matrix associated to this linear transformation is the Vandermonde matrix
of the conjugates $\eta_k$'s. This matrix is invertible because $\eta$ is
separable so that the $\eta_i$ are all distinct reals. Then, inverting the
linear change of variable, we prove that the $a_i$ are also in
$O_\eta(\sqrt{q})$ since the matrix norm of the inverse of the Vandermonde
matrix only depends on $\eta$. This proves the bound given in the proposition
and gives a bijection between $a_i$'s and $\psi_k$'s. We must now justify that
there is a one-to-one correspondance between the $a_i$'s and the coefficients
of $\chi_\pi$.

Since the $\psi_k$ are exactly the real roots (possibly in the Galois-closure
of $\Q(\eta)$) of $\chi_{\psi}$, by Vieta's formula they satisfy the $g$
equations \[s_{g-i}=(-1)^iS_i(\psi_1,\ldots,\psi_g) \text{ for $1\le i \le
g$,}\] where the $S_i$'s are the elementary symmetric polynomials in $g$
variables. Thus, once the $a_i$'s are known, the values for $\psi$ and its
conjugates are known and a unique value for each $s_i$ is deduced. Furthermore,
the expressions of the $s_i$'s in terms of the $\sigma_i$'s form a linear
triangular system whose determinant equals $1$, so that there is an efficiently
computable one-to-one correspondence between $\chi_{\psi}$ and $\chi_{\pi}$.
Therefore, $\chi_\pi$ is uniquely determined by the $a_i$'s.
\end{proof}

Our algorithm is based on determining the $a_i$'s modulo $\ell$ for sufficiently
many $\ell$ until they are known without ambiguity and we can deduce $\chi_\pi$.
While the Weil bounds on the $\sigma_i$'s are enough for our purpose, we have
proven that the $a_i$'s are in $O_\eta(\sqrt{q})$ as in genus 2 and
3~\cite{gaudry2011counting,AGS18}. The next section details the process of recovering such
modular information on the $a_i$'s.

\subsection{Overview of our algorithm}\label{sec:algormanyg}
The general RM point counting algorithm is Algorithm~\ref{algo:rmanyg}.  As
mentioned above, we want to compute the coefficients $a_0,\ldots, a_{g-1}$ of
the endomorphism $\psi$. More precisely, we compute their values modulo
sufficiently many totally-split primes $\ell$ until we can deduce their values
from the bounds of Prop~\ref{prop:chipsi} and the Chinese Remainder Theorem.
Then, the coefficients of $\chi_\pi$ are deduced from the $a_i$'s.

Apart from being totally split, we actually require that the primes $\ell$ we consider satisfy the additional conditions (C1) to (C4) below. The reasons why we need these technical conditions will be made clearer in this section.
\begin{enumerate}
  \item[(C1)] $\ell$ must be different from the characteristic of the base field;
  \item[(C2)] $\ell$ must be coprime to the discriminant of the minimal
    polynomial of $\eta$;
  \item[(C3)] there must exist $\alpha_i\in\mathfrak p_i$ as in
    Lemma~\ref{lem:smallgenanyg} below with norm non-divisible by $\ell^3$ for
    $i\in\{1,\ldots g\}$;
  \item[(C4)] the ideal $\ell\Z[\eta]$ must split completely.
\end{enumerate}

Condition (C2) means that $\Z[\eta]$ is locally maximal at $\ell$, this implies
that even if $\Z[\eta]$ is not the maximal order of $\Q(\eta)$, this defect will
have no impact on the factorization of $\ell$ as a product of prime ideals,
which is at the heart of our complexity gain. Condition (C3) implies that there
is a unique subgroup of order $\ell^2$ in $J[\alpha_i]$.

\begin{algorithm}[ht] \label{algo:rmanyg}
  \Input{$q$ an odd prime power, and $f\in\mathbb F_q[X]$ a monic squarefree
  polynomial of degree $2g+1$ such that the hyperelliptic curve $Y^2=f(X)$ has
explicit RM by $\Z[\eta]$.}
  \Output{The characteristic polynomial $\chi_\pi\in\mathbb Z[T]$ of the
    Frobenius endomorphism on the Jacobian $J$ of the curve.}
 
  $w\gets 1$\;
  Define $C_g$ as in Prop.~\ref{prop:chipsi}\;
  \While{$w\leq 2\, \left[\mathcal{O}_{\Q(\eta)}:\Z[\eta]\right] C_\eta\,\sqrt{q} + 1$} {
    Pick the next prime $\ell$ that satisfies conditions (C1) to (C4);

    Compute the ideal decomposition
    $\ell\, \Z[\eta]=\gothp_1\cdots\gothp_g$, corresponding to the
    eigenvalues $\mu_1,\ldots,\mu_g$ of $\eta$ in
    $J[\ell]$ \;

    \For{$i\gets1$ \KwTo $g$}{
      Compute a small element $\alpha_i$ of $\gothp_i$ as in
      Lemma~\ref{lem:smallgenanyg}\;

      Compute a non-zero element $D_i$ of order $\ell$ in $J[\alpha_i]$ \;

      Find the unique $k_i\in\Z/\ell\Z$ such that
      $k_i\pi(D_i) = \pi^2(D_i) + qD_i$ \;
    }

    Find the unique tuple $(a_0,\ldots,a_{g-1})$ in $(\Z/\ell\Z)^g$ such that
    $ \sum_{j=0}^{g-1}a_j\mu_i^j = k_i$, for $i$ in $\{1,\ldots,g\}$ \;

    $w \gets w\cdot\ell$\;
  }
  Reconstruct $(a_0,\ldots,a_{g-1})$ using the Chinese Remainder Theorem \;

    Deduce $\chi_{\pi}$ from $\psi$.

    \caption{Overview of our RM point-counting algorithm}
\end{algorithm}
The first 3 conditions eliminate only a finite number of $\ell$'s that depends
only on $\eta$. This is not immediate for Condition (C3), so we give
further details. Let us consider an element $\alpha_i$ as in
Lemma~\ref{lem:smallgenanyg} below, i.e. an element represented as a degree
$<g$ polynomial in $\eta$ with integer coefficients in $O_\eta(\ell^{1/g})$. Then
from this size constraint, its norm in the CM field is of the form
$c\,\ell^2$, with $c$ only depending on $\eta$. Therefore, for $\ell >
c$, it is impossible that this norm could be divisible by $\ell^3$. Condition (C3) will thus only discard primes smaller than $c$, which are in finite number.

Condition (C4) will eliminate a finite proportion of primes $\ell$
depending only on $\eta$ which we need to detail in order to bound the size of
the largest $\ell$ that our algorithm will have to consider. Given a genus-$g$
curve $\C$ with RM by $\Z[\eta]$, by Chebotarev's density theorem, the
proportion of primes $\ell$ satisfying the last condition is at least
$1/\#\Gal(\Q(\eta)/\Q)$, which is bounded below by $1/g!$. To count points on
$\C$, we need to find $L$ a set of primes satisfying all the above conditions
and such that $\prod_{\ell\in L}\ell >
2\left[\mathcal{O}_{\Q(\eta)}:\Z[\eta]\right] C_\eta\sqrt{q}$. By the prime
number theorem, both the number and size of the primes contained in $L$ are in
$O(g!\log (C_\eta q))$. In some particular cases, the proportion of \enquote{nice}
primes may be much larger: for instance when the RM field is the totally real
subfield of a cyclotomic field. In the field $\Q(\zeta_n+\zeta_n^{-1})$, a
prime $\ell$ totally splits if and only if $\ell\equiv \pm 1 \bmod n$, and
therefore condition (C4) is satisfied by a proportion of primes equal to
$2/(n-1)=1/g$. In that case, the number and size of primes in the set $L$ can
be reduced to $O(g\log (C_\eta q))$.

We now explain how the algorithm works for a given split $\ell$. First its
decomposition as a product of prime ideals $\ell\,
\Z[\eta]=\gothp_1\cdots\gothp_g$ is computed, and for each prime ideal
$\gothp_i$, a non-zero element $\alpha_i$ in $\gothp_i$ is found with a small
representation as in Lemma~\ref{lem:smallgenanyg} below. In fact, $\gothp_i$ is
not necessarily principal and $\alpha_i$ need not generate $\gothp_i$. The
kernel of $\alpha_i$ is denoted by $J[\alpha_i]$ and it contains a subgroup
isomorphic to $\Z/\ell\Z\times\Z/\ell\Z$, since the norm of $\alpha_i$ is a
multiple of $\ell$. 

Since $\ell$ satisfies Condition (C2), there is a correspondence between the
prime ideals $\gothp_i$ in the decomposition of $\ell$ in $\Z[\eta]$ and
irreducible factors of the minimal polynomial of $\eta$ modulo $\ell$. On one
side we have $g$ prime ideals coprime to each other, so we have $g$ coprime
factors $X-\lambda_i$ on the other side. 

This way, we know that the two-element representation of each ideal $\gothp_i$
is of the form $(\ell,\eta-\lambda_i)$, which means that $\lambda_i$ is an
eigenvalue of $\eta$ viewed as an endomorphism of $J[\ell]\cong(\mathbb
Z/\ell\mathbb Z)^{2g}$. Furthermore, the $\gothp_i$'s are coprime so we have
$J[\ell]=\bigoplus_{i=1}^g J[\gothp_i]$ and since each of them has norm $\ell$
in $\Q(\eta)$, they have norm $\ell^2$ in the CM field and therefore each
$J[\gothp_i]$ is isomorphic (as a group) to $\left(\Z/\ell\Z\right)^2$. Since
$\ell$ satisfies Condition (C3) we know that it is actually the only subgroup
of $J[\alpha_i]$ isomorphic to $\left(\Z/\ell\Z\right)^2$.

On $J[\gothp_i]\subset J[\alpha_i]$, the endomorphism $\eta$ acts as the
multiplication by $\lambda_i$. Therefore, the endomorphism
$\psi=\sum_{i=0}^{g-1} a_i\eta^i$ also acts as a scalar multiplication on this
2-dimensional space, and we write $k_i\in\Z/\ell\Z$ the corresponding
eigenvalue: for any $D_i$ in $J[\gothp_i]$, we have $\psi(D_i) = k_iD_i$. On
the other hand, from the definition of $\psi$, it follows that $\psi\pi =
\pi^2+q$. Therefore, if such a $D_i$ is known, we can test which value of $k_i
\in \Z/\ell\Z$ satisfies
\begin{equation}\label{eq:pirm}
 k_i\pi(D_i) = \pi^2(D_i) + qD_i.
\end{equation}

Since $\ell$ is a prime and $D_i$ is of order exactly $\ell$, this is also the
case for $\pi(D_i)$. Finding $k_i$ can then be seen as a discrete logarithm
problem in the subgroup of order $\ell$ generated by $\pi(D_i)$; hence the
solution is unique. Equating the two expressions for $\psi$, we get explicit
relations between the $a_j$'s modulo $\ell$: $$ \sum_{j=0}^{g-1}a_j\lambda_i^j
\equiv k_i \bmod \ell.$$ Therefore we have a linear system of $g$ equations in
$g$ unknowns, the determinant of which is the Vandermonde determinant of the
$\lambda_i$, which are distinct since the ideals $\gothp_i$ are coprime.
Hence the system can be solved and it has a unique solution modulo $\ell$.

It remains to show how to construct a divisor $D_i$ in $J[\gothp_i]$, i.e. an element of
order $\ell$ in the kernel $J[\alpha_i]$.  Since an explicit expression of
$\eta$ as an endomorphism of the Jacobian of $\C$ is known, an explicit
expression can be deduced for $\alpha_i$, using the explicit group law. The
coordinates of the elements of this kernel are solutions of a polynomial system
that can be directly derived from this expression of $\alpha_i$, using a
modelling similar to that of~\cite{AGS17}. Likewise, we use the
geometric resolution algorithm to find the solutions of this system, perhaps in a
finite extension of the base field, from which divisors in $J[\alpha_i]$ can be
constructed.  Multiplying by the appropriate cofactor, we can reach all the
elements of $J[\gothp_i]$; but we stop as soon as we get a non-trivial one.

\begin{lemma}\label{lem:smallgenanyg}
    For any prime $\ell$ that splits completely in $\Z[\eta]$, each prime ideal
    $\gothp$ above $\ell$ contains a non-zero element $\alpha$ of the form
    $\alpha = \sum_{i=0}^{g-1}\alpha_i\eta^i$, where the $|\alpha_i|$  are
    integers smaller than $i_\eta^{-1/g}\ell^{1/g}$ and $i_\eta$ stands for the index $\left[\OO_{\Q(\eta)}:\Z[\eta]\right]$. 
\end{lemma}

\begin{proof}
 
  In this proof, we will consider two different embeddings from $\gothp$ to
  $\R^g$. First, the elements of the ideal $\gothp$ are inside $\Z[\eta]$ so we
  can define the embedding $\tau:\sum_{j=0}^{g-1}a_{j}\eta^j \mapsto
  (a_0,\ldots a_{g-1})$. Let $(\beta_1,\ldots,\beta_g)$ a basis of $\gothp$ as
  a $\Z$-module and define $B$ the $g\times g$ matrix obtained by concatenating
  the $g$ column vectors $\tau(\beta_i)$. The volume of the lattice
  $\tau(\gothp)$ is by definition $|\det B|$. 
  
  To compute this volume, we introduce another embedding from $\gothp$ to a
  lattice of $\R^g$ whose volume is already known. Let us introduce the $g$
  real embeddings $\sigma_i:\Q(\eta)\rightarrow \R$ and define the embedding
  $\sigma:\Q(\eta)\rightarrow \R^g$ by $\sigma:\alpha\mapsto
  (\sigma_1(\alpha),\ldots,\sigma_g(\alpha))$.  By~\cite[Prop~4.26]{milneant},
  $\sigma(\gothp)$ is a full lattice in $\R^g$ and its volume is
  $N(\gothp)\sqrt{|\Delta|}$, where $\Delta$ is the discriminant of
  $\Q(\eta)$.  In our case, $N(\gothp)=\ell$ so the volume of $\sigma(\gothp)$
  is $\ell\sqrt{|\Delta|}$.

  Now, we link both volumes by remarking that we can reorder the embeddings to
  make them compatible with our previous definition of the conjugates $\eta_k$
  of $\eta$ so that for any $i,k\in\{1,\ldots,g\}^2$ we have
  $\sigma_k(\beta_i)=\sum_{j=0}^{g-1}b_{ij}\eta_k^j$, where the $b_{ij}$
  satisfy $\beta_i=\sum_{j=0}b_{ij}\eta^j$. Phrased differently, with $V$ the
  Vandermonde matrix of the conjugates of $\eta$, this amounts to
  $\sigma_k(\beta_i)$ being the $k$-th coordinate of the vector
  $V\tau(\beta_i)$. If we call $S$ the matrix whose entries are the
  $\sigma_{j}(\beta_i)$ then we have $S=VB$.  We know that $|\det S|$ is the
  volume of $\sigma(\gothp)$ which we previously computed so we deduce $|\det
  B|=\ell\sqrt{|\Delta|}/|\det V|$.

  Finally, we give a significance to the quotient $\sqrt{|\Delta|}/|\det V|$.
  Defining $i_\eta$ as in the statement of the lemma, we have
  $\Disc(1,\eta,\ldots,\eta^{g-1})=i_\eta^2\Disc(\OO_{\Q(\eta)}/\Z)$ (see for
  instance~\cite[Remark~2.25]{milneant}) but since
  $\Disc(1,\eta,\ldots,\eta^{g-1})=(\det V)^2$ we finally conclude that the
  volume of $\tau(\gothp)$ is $\ell /i_\eta$.
 
 Let us consider $C=\{x\in\mathbb{R}^{g}\mid ||x||_{\infty}\le
 i_\eta^{-1/g}\ell^{1/g}\}$.  The volume of the convex $C$ is $2^g\ell/i_\eta$.
 Since $g$ is the dimension of $\tau(\gothp)$ and $\ell /i_\eta$ is its volume,
 Minkowski's theorem guarantees the existence of a non-zero element $v$ of
 $\tau(\gothp)$ belonging to $C$. By definition, $v=\sum_{i=0}^{g-1}v_i\eta^i$
 is an element of $\gothp$ whose coordinates $v_i$'s are integers of absolute
 values bounded by $i_\eta^{-1/g}\ell^{1/g}$, which concludes the proof.
  \end{proof}

  Since we know it exists, given one of the ideals $\gothp_i$, we can find
  $\alpha_i$ a small element of $\gothp_i$ as in Lemma~\ref{lem:smallgenanyg} by
  exhaustive search in at most $2^g\ell/i_\eta$ operations in $\Z[\eta]$. Note that
  there is an extensive litterature on finding short vectors in a lattice of
  dimension $d$, motivated for instance by cryptographic applications. An
  example is the quantum algorithm of~\cite{cdpr} which computes a
  $2^{\softO(\sqrt{d})}$-approximation of the shortest non-zero vector in time
  polynomial in $d$. Restricting to classical algorithms, the best option in
  general is the BKZ algorithm~\cite{bkz} that computes a
  $2^{\softO(d^{\alpha})}$-approximation in time $2^{\softO(d^{1-\alpha})}$,
  for any $\alpha\in [0,1]$. In our case however, the existence of a very short
  vector is already known and, more importantly, the factor $2^g$ due to the
  dimension is acceptable since it vanishes in the $O_\eta$-notation.

\section{Modelling kernels of endomorphisms}\label{sec:cplxrmanyg}

Let $\alpha$ be an explicit endomorphism of degree $O(\ell^2)$ on the Jacobian
of $\C$, which satisfies the properties of Lemma~\ref{lem:smallgenanyg}.  We
want to compute a polynomial system that describes the kernel $J[\alpha]$ of
$\alpha$, and then solve it. The resultant-based approach of~\cite{AGS18}
cannot be used as the degrees are squared each time we eliminate a variable,
causing an exponential dependency in $g$ in the exponent of $\ell$. Instead, we
use the modelling techniques from~\cite{AGS17}, where the endomorphism $\alpha$
replaces the multiplication by $\ell$. This time, the $g$ variables of large
degrees have degrees in $O_\eta(\ell^{3/g})$ instead of $O_\eta(\ell^3)$ so
that the final complexity bound for computing the kernel $\alpha$ is in
$O_\eta(\ell^D(\log q)^2)$ binary operations, with $D$ an absolute constant. 

The main change between this section and~\cite[Sec.~4 \& 5]{AGS17} is that the
$d_i$ and $e_i$ no longer denote $\ell$-division but $\alpha$-division
polynomials, and the polynomials $u_j$ and $v_j$ intervening in the Mumford
representation of the candidate kernel element are modified accordingly. The
structure of our modelling is very similar but require some adaptations at
various places, which is the reason why we repeat the analysis in the generic
case. In the non-generic case, we go over the main results
of~\cite[Sec.~5]{AGS17} and detail the parts requiring adjustments.

\subsection{The generic case}\label{sec:genanyg}

Let us first recall the definition of Cantor's $\ell$-division polynomials introduced in~\cite{Ca94}, the
coefficients of the polynomials
$\delta_{\ell}(X)$ and $\varepsilon_{\ell}(X)$ such that, for $(x,y)$ a generic
point of the curve and $\ell > g$, we have
\begin{equation*}
    \ell \left(
    (x,y)-P_{\infty}\right)=\left\langle\delta_{\ell}\left(\frac{x-X}{4y^2}
    \right),\varepsilon_{\ell}\left(\frac{x-X}{4y^2}
    \right)\right\rangle.
\end{equation*}

An important step towards our complexity bounds is to bound the degrees of these
polynomials, so that we can later on deduce degree-bounds for the polynomial
systems modelling the kernels $J[\alpha_i]$. To this end, we use the following result
proven in~\cite[Sec.~6]{AGS17}.

\begin{theorem}{\cite[Lemma 10]{AGS17}}\label{lem:deg_delta} 
For any integer $\ell >g$, the polynomial $\delta_{\ell}(X)$
of degree $g$ in $X$ has coefficients in $\F_q[x]$ whose degrees in $x$
are bounded by $g\ell^3/3+\Og(\ell^2)$; the polynomial
$\varepsilon_{\ell}(X)/y$ has coefficients in $\F_q(x)$ whose respective numerators and denominators have degrees bounded by
$2g\ell^3/3+\Og(\ell^2)$.  Furthermore, the roots of the denominators are
roots of the leading coefficient of $\delta_\ell(X)$. \end{theorem}

These polynomials describe the multiplication by $\ell$, but for our purpose we
need to describe more general endomorphisms of $\Jac\C$, i.e. endomorphisms
corresponding to an element $\alpha$ of $\Z[\eta]$. Thus, we define the
$\alpha$-division polynomials $d_i$ and $e_i$ such that, denoting by $P=(x,y)$
the generic point of $\C$, a non-normalized Mumford form of
$\alpha(P-{P_{\infty}})$ is equal to \[\left\langle \sum_{i=0}^g
d_i(x)X^i,y\sum_{i=0}^{g-1}\frac{e_i(x)}{e_g(x)}X^i \right\rangle.\] 

By Lemma~\ref{lem:smallgenanyg}, we know that
$\alpha=\sum_{i=0}^{g-1}\alpha_i\eta^i$ with $| \alpha_i| =O_\eta(\ell^{1/g})$.
Since the degrees of the $\eta^i(P-{P_{\infty}})$ do not depend on $\ell$, by
Theorem~\ref{lem:deg_delta} applied to Cantor's $\alpha_i$-division polynomials
we prove that the degrees of the $d_i$'s and $e_i$'s are in $O_\eta(\ell^{3/g})$.

\begin{definition}\label{def:alphadivgen} 
    In what follows, we will say that an element of $J$ is 
    $\alpha$-generic if it has weight $g$ and the corresponding
    reduced divisor $\sum_{i=1}^g (P_i-P_{\infty})$ satisfies the following two
    properties:
    \begin{itemize}
        \item For any $i$, the $u$-coordinate of the divisor
            $\alpha(P_i-P_{\infty})$ in Mumford form has   
	    degree $g$;
        \item For any $i\ne j$, the $u$-coordinates of the
            divisors $\alpha(P_i-P_{\infty})$ and $\alpha(P_j-P_{\infty})$
            are coprime.
    \end{itemize}
    
    This implies that if an affine point $P$ occurs in the support of 
    $\alpha(P_i-P_{\infty})$ then neither $P$ nor $-P$ appears in the support
    of another $\alpha(P_j-P_{\infty})$.
\end{definition}
  Suppose there exists $D=\sum_{i=1}^g(P_i-P_{\infty})$ an $\alpha$-generic
  divisor in $J$. We shall consider a system equivalent to $\alpha(D)=0$ but
  let us first introduce some notation. For each point $P_i=(x_i,y_i)$ in the
  support of $D$, we denote $\langle u_i,v_i \rangle$ the Mumford form of
  $\alpha(P_i-P_{\infty})$ and $(a_{ij},b_{ij})_{1\leq j\leq g}$ the
  coordinates of the $g$ points in its support counted with multiplicities,
  which means that for any $i$ the $g$ roots of $u_i$ are exactly the $a_{ij}$,
  and that for any $j$, $b_{ij}=v_i(a_{ij})$.

  \begin{proposition}\label{prop:modelanyg}
     We can model the set of generic $\alpha$-division elements as the solution
     set of a bihomogeneous polynomial system consisting of $O(g^2)$ equations
     in $\F_q[X_1,\ldots,X_g,Y_1,\ldots,Y_{n_y}]$ such that $n_y=O(g^2)$ and
     the degrees $d_x$ and $d_y$ in the $X_i$'s and $Y_j$'s are respectively in
     $O_\eta(\ell^{3/g})$ and $O_\eta(1)$.
  \end{proposition} 
  
  \begin{proof}
  Following the modelling of~\cite[Sec.~4]{AGS17}, we have $\alpha(D)=0$ if
  and only if the sum of the divisors $\sum_{i=1}^g\alpha(P_i-P_{\infty})$ is a
  principal divisor. The only pole is at infinity, so this is equivalent to the
  existence of a non-zero function $\varphi\in\mathbb F_q(\C)$ of the form
  $P(X)+YQ(X)$ with $P$ and $Q$ two polynomials such that the $g^2$ points
  $(a_{ij},b_{ij})$ are the zeros of $\varphi$, with multiplicities. Since we
  want $\varphi$ to have $g^2$ affine points of intersection with the curve
  $\mathcal{C}$ (once again, counted with multiplicities), the polynomial
  $\Res_Y(Y^2-f,P+YQ)=P^2-fQ^2$ must have degree $g^2$ which yields $2\deg(P)\le
  g^2$ and $2\deg(Q)\le g^2-2g-1$. Exactly one of those two bounds is even (it
  depends on the parity of $g$), and for this particular bound, the inequality
  must be an equality, otherwise the degree of the resultant would not be $g^2$.
  Since the function $\varphi$ is defined up to a multiplicative constant, we
  can normalize it so that the polynomial $P^2+fQ^2$ is monic, which is
  equivalent to enforce that either $P$ or $Q$ is monic depending on the parity
  of $g$.

  For a fixed $i\in[1,g]$, requiring the $(a_{ij},b_{ij})$ to be zeros of
  $\varphi$ amounts to asking for the $a_{ij}$ to be roots of $P(X)+Q(X)v_i(X)$,
  with multiplicities. Since the $a_{ij}$ are by definition the roots of the
  $u_i$, $\alpha(D)=0$ is equivalent to $g$ congruence relations $P+Qv_i
  \equiv 0 \bmod u_i$.  Thus, for any $\alpha$-generic divisor, $\alpha(D)=0$ is
  equivalent to the existence of $P$ and $Q$ satisfying the above $g$ congruence
  relations.
 
  The variables are the coefficients of $P$ and $Q$, as well as the $x_i$ and
  $y_i$. With the degree conditions and the normalization, we have $g^2-g$
  variables coming from $P$ and $Q$.  Adding the $2g$ variables $x_i$ and $y_i$,
  we get a total of $g^2+g$ variables.  Each one of the $g$ congruence relations
  amounts to $g$ equations providing a total of $g^2$ conditions on the
  coefficients of $P$ and $Q$. The fact that the $(x_i,y_i)$ are points of the
  curve yields the $g$ additional equations $y_i^2=f(x_i)$.  Finally, we have to
  enforce the $\alpha$-genericity of the solutions, which can be done by
  requiring that $\prod_i d_g(x_i)e_g(x_i) \prod_{i<j}\Res(u_i,u_j)\not=0$. Note
  that we do not extend Theorem~\ref{lem:deg_delta} to the $\alpha$-division
  polynomials but instead add the
  non-vanishing condition for the denominator $e_g$ of the $v$-coordinate of
  $\alpha(D)$. Still, we get a polynomial system with $g^2+g$ equations in
  $g^2+g$ variables, together with an inequality. 
  
  We now estimate the degrees to which the variables occur in the equations.
  Each congruence relation is obtained by reducing $P+Qv_i$, which is a
  polynomial of degree $O(g^2)$ in $X$, by $u_i$ which is of degree $g$.  We can
  do it by repeatedly replacing $X^g$ by $-\sum_{j<g} (d_j(x_i)/d_g(x_i)) X^j$,
  which we will have to do at most $O(g^2)$ times. Since the $d_j$ have degree
  in $O_\eta(\ell^{3/g})$ in $x_i$, the fully reduced polynomial will have
  coefficients that are fractions for which the degrees of the numerators and of
  the denominators are at most $O_\eta(\ell^{3/g})$ in the $x_i$ variables. In
  these equations, the degree in the $y_i$ variables and in the variables for
  the coefficients of $P$ and $Q$ is 1. The degrees in $x_i$ and $y_i$ in the
  curve equations are $2g+1$ and $2$ respectively.

  It remains to study the degree of the inequality. Each resultant is the
  determinant of a $2g\times 2g$ Sylvester matrix whose coefficients are the
  $d_i$, which have degrees bounded by $O_\eta(\ell^{3/g})$. Since for any $i$
  there are exactly $g$ resultants involving $x_i$ in the product, the degree
  of this inequality in any $x_i$ is in $O_\eta(\ell^{3/g})$, and it does not
  involve the other variables. In order to be able to use
  Proposition~\cite[Prop.~3]{AGS17} that we recall in
  Section~\ref{sec:cplxrmangyg}, we must model this inequality by an equation,
  which is done classically by introducing a new variable $T$ and by using the
  equation $T\cdot\prod_i d_g(x_i)e_g(x_i) \prod_{i<j}\Res(u_i,u_j)=1$.
  
  To conclude, we have a polynomial system with two blocks of variables: the $g$
  variables $x_i$ on the one hand and the $g^2-g$ variables coming from the
  coefficients of $P$ and $Q$, along with the $g$ variables $y_i$ on the other
  hand. The degrees of the equations in the first block of variables grows
  cubically in $\ell^{1/g}$, while the degrees in the other block of variables
  depends only on $\eta$. 
\end{proof}

\subsection{Non-generic kernel elements}\label{sec:nongenanyg}

As in~\cite[Sec.~4]{AGS17}, apart from the neutral element, we expect to
capture the whole kernel of the endomorphism $\alpha$ by using the modelling of
Section~\ref{sec:genanyg}. Contrary to~\cite{AGS17},
Algorithm~\ref{algo:rmanyg} does not require us to find a basis of
$J[\alpha]$ because the determination of the $k_i$'s does only require a single
non-zero element in each $J[\alpha_i]$. Thus, a study of non-generic elements
in $J[\alpha]$ is necessary only if there is no $\alpha$-generic element in
$J[\alpha]$. Such a case happens if and only if the polynomial $\prod_{i=1}^g
d_g(x_i)e_g(x_i)\prod_{i\ne j}\Res(u_i,u_j)$ in the variables $x_1,\ldots,x_g$
vanishes on $J[\alpha]$.  It seems very unlikely that the whole set $J[\alpha]$
would live in such a hypersurface, and if it happens, one can discard the
$\ell$ for which we fail to find an $\alpha$-generic element.  Although it
seems even more unlikely that this situation could happen for sufficiently many
$\ell$ so as to threaten the validity of our complexity bound, we are far from
a proven statement and do not exclude it might be possible to design a highly
non-generic curve providing a counterexample.

Therefore, we follow the non-genericity analysis of~\cite[Sec.~5]{AGS17} except
that we consider $u_i$ and $v_i$ defined as the Mumford form of
$\alpha(P_i-P_\infty)$ instead of $\ell (P_i-P_\infty)$. Let us first briefly review
the non-generic situations that one can encounter,
following~\cite[Sec.~5.1]{AGS17} and keeping the same numbering.

\paragraph{Case 1: Modelling a kernel element of weight $w<g$.} 

    We write $D=\sum_{i=1}^w (P_i-P_{\infty})$ and look for a $\varphi=P(X)+YQ(X)$
    vanishing at each point of each reduced divisor $\alpha(P_i-P_{\infty})$.
    This is similar to the Case 1 of~\cite[Sec.~5.1]{AGS17}.

\paragraph{Case 2: Modelling a kernel element with multiple points.} 
    
It may happen that the element we are looking for is $D=\sum_{i=1}^w
    (P_i-P_{\infty})$ but not all the $P_i$'s are distinct. In that case, we
    rewrite it $D=\sum_{j=1}^s \mu_j (P_j-P_{\infty})$ such that the
    $P_j$'s are distinct and look for a $\varphi=P(X)+YQ(X)$ vanishing at each
    point of each reduced divisor $\mu_j\alpha(P_j-P_{\infty})$. Apart from
    the modification of $u_i$ and $v_i$, the modelling is identical to that
    of~\cite{AGS17}.

\paragraph{Case 4: Modelling a kernel element after reduction.} 

    Even if all the $\alpha(P_i-P_{\infty})$ had full weight, there may still
    be less than $g^2$ points in the union of their supports due to possible
    cancellations of points appearing in the supports of several
    $\alpha(P_i-P_{\infty})$ with different signs. Exactly as
    in~\cite[Sec.~5.1]{AGS17}, if $P$ appears within $\alpha(P_i-P_{\infty})$
    and $\alpha(P_j-P_{\infty})$ with respective multiplicities $\nu_i$ and
    $\nu_j$ of opposite signs, this is modelled by ensuring that the
    corresponding $u_i$, $u_j$, and $v_i+v_j$ share a common factor
    $(X-\xi)^{\nu}$ where $\nu=\max(|\nu_i|,|\nu_j|)$. In that case, we look
    for $\varphi(X,Y)=(X-\xi)^\nu(\widetilde{P}(X)+Y\widetilde{Q}(X))$, with
    $\widetilde{P}$ coprime to $\widetilde{Q}$. Once modified the values of the
    $u_i$ and $v_i$, nothing changes from~\cite{AGS17}. 

\paragraph{Case 5: Modelling a kernel element with multiplicity.} 

   Conversely, $\alpha(P_i-P_{\infty})$ and $\alpha(P_j-P_{\infty})$ can also
   share the same point with multiplicities of identical sign, leading to
   multiplicities in the reduced divisor $\alpha(D)$. Similarly to what was done
   in the Case 5 of~\cite[Sec.~5.1]{AGS17}, we can group the corresponding
   $u_i$, $u_j$, $v_i$ and $v_j$ in polynomials $U$ and $V$ such that $U | V^2
   -f$ and $\deg V < \deg U$, and then look for $\varphi=P(X)+YQ(X)$ such that
   $P+QV \equiv 0 \bmod U$. Once again, nothing changes apart from the
   definition of the $u_i$'s and $v_i$'s.

\paragraph{Case 3: Low weight after applying $\alpha$.} 
   
    We kept this case for the end because it is not a straightforward extension
    of the Case 3 appearing in~\cite[Sec.~5.1]{AGS17}.  Until now, we assumed
    that all the $P_i$'s in the support of $D$ were such that
    $\alpha(P_i-P_{\infty})$ had weight $g$, i.e. $d_g(x_i)\ne 0$. We now want
    to model the case where $D=\sum_{i=1}^w (P_i-P_{\infty})$ such that each
    $\alpha(P_i-P_{\infty})$ has weight $w_i$. In~\cite{AGS17}, this was done
    using a result from~\cite{Ca94} giving a necessary and sufficient condition
    for $\ell(P_i-P_{\infty})$ to be of weight $w_i$. When $\alpha$ is an
    endomorphism other than scalar multiplication, no such result holds a
    priori. In what follows, we address this issue by designing non-generic
    $\alpha$-division polynomials (Definition~\ref{def:nongenalphadiv} below)
    $\Gamma_{\alpha,t}$ and ${\Delta}_{\alpha,t}$ such that $\alpha\left(
    (x,y)-P_{\infty} \right)$ has weight $w$ if and only if
    $\Delta_{\alpha,w}(x)=0$ and $\Gamma_{\alpha,w-1}(x)\ne 0$.

\paragraph{Combining all degeneracies.}

As in~\cite[Sec.~5.2]{AGS17}, we have to consider situations in which several
of the previous cases occur simultaneously. Note that while we wanted to
compute the whole $\ell$-torsion in~\cite{AGS17}, we now only need one kernel
element per endomorphism $\alpha_i$ to determine $\chi_\pi\bmod\ell$.
Therefore, after finding a non-zero solution to any of the subsequent systems,
one need not consider the others. The aim of the Section is to
prove Proposition~\ref{prop:modelnganyg} below, in order to bound the number and
respective sizes (number of equations and variables) of all the systems
modelling non-generic situations.

  \begin{proposition}\label{prop:modelnganyg}
    We can model the set of non-generic elements of $J[\alpha]$ as the solution
    set of $O_\eta(1)$ bihomogeneous polynomial systems each consisting of
    $O(g^2)$ equations in $\F_q[X_1,\ldots,X_g,Y_1,$ $\ldots,Y_{n_y}]$ such
    that $n_y=O(g^2)$ and the degrees $d_x$ and $d_y$ in the $X_i$'s and
    $Y_j$'s are respectively in $O_\eta(\ell^{3/g})$ and $O_\eta(1)$.
  \end{proposition}

To do so, we first describe a data structure to represent any combination of
the non-generic cases detailed above. Then, we explain how we can transform any
occurrence of this data structure into a polynomial system. Throughout this
transformation, we will keep track of equations and variables that we need and
sum everything up in Tables~\ref{tab:variables} and~\ref{tab:degrees}.
Everything here is a careful adaptation of~\cite[Sec.~5.2]{AGS17} with three
notable differences: the fact that we consider $\alpha(D)$ instead of $\ell D$,
the slight difference in defining non-generic $\alpha$-division polynomials and
most importantly the fact that the degrees $d_x$ are now in
$O_\eta(\ell^{3/g})$ instead of $O_g(\ell^3)$. Any reader convinced by this
very brief overview can skip to these tables to avoid technicalities. 

\paragraph{A data structure to describe each type of non-genericity.}

We consider an $\alpha$-torsion divisor $D$ of weight $w\le g$ (like in Case
1). Next, a partition $\mu=(\mu_1,\ldots,\mu_k)$ of $w$ is picked
to represent the multiplicity pattern in the $u$-coordinate of the
$\ell$-torsion divisor, as in Case~2 so that
$D=\sum_{i=1}^k\mu_i(P_i-\infty)$. Then, a vector $t=(t_1,\ldots,t_k)$ is
chosen, to represent the weights of the $P_i$ after applying
$\mu_i\,\alpha$ as in Case~3: for $i$ in $[1,k]$, the reduced divisor
$\mu_i\,\alpha(P_i-\infty)$ is of weight $t_i$. Then, we need to consider
how many common or opposite points these divisors have in their supports to take
into account Cases~4 and~5. We denote by $Q_1,\ldots,Q_s$ the points in
the union of the supports of all the reduced divisors
$\mu_i\,\alpha(P_i-\infty)$, keeping only one point in each orbit under the
hyperelliptic involution. We represent the non-genericity by a $k\times s$
matrix $M$ such that its non-zero entries $m_{ij}$ verify
$m_{ij}=\ord_{Q_j}(\mu_i\,\alpha(P_i-\infty))$ when $Q_j$ is in the support
of $\mu_i\,\alpha(P_i-\infty)$ or
$m_{ij}=-\ord_{Q_j'}(\mu_i\,\alpha(P_i-\infty))$ when the hyperelliptic
conjugate $Q_j'$ of $Q_j$ is in the support. Note that this matrix, that we
shall call the matrix of shared points, represents both multiplicities and
non-semi-reduction. Since the row $i$ represents what happens with points in
the support of $\mu_i\,\alpha(P_i-\infty)$, which is of weight $t_i$, the
sum of the absolute values of the entries of the row $i$ of $M$ is equal to
$t_i$. Also, by construction, there is at least one non-zero
entry in each column. An additional complication arises when one of the $P_i$ is a
ramification point, i.e. when its $y$-coordinate is zero, because this would
cause multiplicities if care is not taken, leading to non-radicality of the
polynomial system we build. Since this corresponds to $P_i-\infty$ being of
order 2, the weight $t_i$ is equal to 0 or 1. If $t_i=0$, then the divisor
$D-\mu_i(P_i-\infty)$ is also an $\alpha$-torsion divisor of weight
$w-\mu_i$, so that we can reconstruct $D$ from another polynomial system.
There is however no obvious way to avoid the possibility $t_i=1$. Therefore,
we will encode the fact that $P_i$ is a ramification point by a bit
$\epsilon_i$ that can be set only in the cases where $t_i=1$ and $\mu_i=1$.
Changing the order of the columns of $M$ amounts to permuting the points $Q_j$.
Also, changing the sign of all the entries of a column $j$ corresponds to
taking the opposite of the point $Q_j$. While it would not change the final
complexity not to do so, it makes sense to consider only normalized
tuples, in the sense that the columns of $M$ are sorted in lexicographical
order, and the choice between a point $Q_j$ and its opposite is done so that
the sum of all elements in the corresponding column is nonnegative. We remark
that this is not enough to guarantee that two normalized tuples do not describe
similar situations.  This is not a problem for the general algorithm: the same
$\alpha$-torsion elements can correspond to solutions of two different systems,
but what is important to us is non-multiplicity (i.e. radicality of the ideal)
in each individual system. All this discussion is summed up by the following definition:

\begin{definition}{\cite[Def.~13]{AGS17}}\label{def:ngt}
    A {\em normalized non-genericity tuple} is a tuple
    $(w,\mu,t,\epsilon, M)$, where
    $1\le w\le g$ is an integer, $\mu=(\mu_1,\ldots,\mu_k)$ is a
    partition of $w$, $t$ and $\epsilon$ are vectors $t=(t_1,\ldots,t_k)$ 
    and $\epsilon=(\epsilon_1,\ldots,\epsilon_k)$ of the same length
    as $\mu$
    with $1\le t_i\le g$ and $\epsilon_i\in\{0,1\}$, where $\epsilon_i$
    can be $1$ only if $t_i=1$ and $\mu_i=1$,
    and finally
    $M$ is
    a matrix with $k$ rows and $s$ columns, where $0\le s\le g\,k$, and
    its entries are integers such that:
    \begin{itemize}
        \item For all $1\le i\le k$, the sum of the absolute values of the entries on the row
            $i$ is equal to $t_i$;
        \item The columns are sorted in lexicographical order;
        \item The sum of the rows of the matrix is a vector whose coordinates
	  are nonnegative.
    \end{itemize}
    \end{definition}

We can follow the analysis of~\cite[Sec.~5.2]{AGS17} to describe more
explicitly the equations and their degrees / number of variables, and remark
that the only part that does not generalize readily is the definition of
non-generic $\alpha$-division polynomials, as in the Case 3 above. Let us first
fix this issue. 

When the weight $t_i$ of $\mu_i\,\alpha(P_i-P_{\infty})$ is strictly smaller
than $g$, the usual coordinate system given by the Mumford form is no longer
available, due to the vanishing of the denominator $e_g(x_i)$. We define an
adequate coordinate system to describe non-generic elements of weight $t$. Let
us consider the variety 
$$V_{\alpha,t}=\left\lbrace (x,y)\in\C\mid \alpha\left( (x,y)-P_{\infty}
\right)\text{ has weight } t\right\rbrace.$$

We want to define polynomials $\Delta_{\alpha,t}$ and $\Gamma_{\alpha,t}$ such
that a point is in $V_{\alpha,w}$ if and only
if $\Delta_{\alpha,w}(x)=0$ and $\Gamma_{\alpha,w-1}(x)\ne 0$
iteratively. First, $\Delta_{\alpha,g-1}=\GCD(d_g,e_g)$, so that the points
$(x,y)$ of $V_{\alpha,g-1}$ satisfy $\Delta_{\alpha,g-1}(x,y)=0$. Assuming
that for $k<g$ we have already constructed a squarefree polynomial
$\Delta_{\alpha,k}$ vanishing on the abscissae of points in $V_{\alpha,k}$, then
one can compute $\alpha\left( (x,y)-P_{\infty} \right)$ over
$\F_p[x,y]/(\Delta_{\alpha,k}(x),y^2-f(x))$. By our induction hypothesis, the
Mumford form of the result is $\langle u,v\rangle$, with $u$ of degree $k$ and
$v$ of degree $k-1$. Let $\Gamma_{\alpha,k-1}$ be the product of $\LC(u)$ with
the denominator of $\LC(v)$, then $V_{\alpha,k}$ is the set of points $(x,y)$
such that $\Delta_{\alpha,k}(x)=0$ and $\Gamma_{\alpha,k-1}(x)\ne 0$.
Furthermore, $\Delta_{\alpha,k-1} =\GCD(\Delta_{\alpha,k},\Gamma_{\alpha,k-1})$
vanishes on the points of $V_{\alpha,k-1}$.

To avoid multiplicities, we replace ${\Delta}_{\alpha,
t}(x)$ by the square-free polynomial whose roots are exactly the roots of
$\Delta_{\alpha, t}(x)$ that are not roots of $\Gamma_{\alpha,t-1}(x)$ when it
is necessary. Note that the degrees of the $\Delta$ and $\Gamma$ are by
construction bounded by $\deg\Delta_{\alpha,g-1}\le \deg d_g$ with $\deg d_g$
itself bounded by $O_\eta(\ell^{1/g})$.
This way, we state an analogue of~\cite[Def.~14]{AGS17} for
non-generic $\alpha$-division polynomials:

\begin{definition}\label{def:nongenalphadiv}
    The non-generic $\alpha$-division polynomials $\frak{u}_{\alpha,t}$ and
    $\frak{v}_{\alpha,t}$ are the polynomials in $X$ with coefficients in
    $\F_p[x,y]/({\Delta}_{\alpha, t}(x), y^2-f(x))$ such that
    $$ \alpha ((x,y)-\infty) = \Big\langle \frak{u}_{\alpha,t}(X),
    \frak{v}_{\alpha,t}(X)\Big\rangle,$$
    in weight-$t$ Mumford representation: $\frak{u}_{\alpha,t}(X)$ is monic
    of degree $t$, $\frak{v}_{\ell,t}(X)$ is of degree at most
    $t-1$ and they satisfy $\frak{u}_{\alpha,t}\, |\, \frak{v}_{\alpha,t}^2-f$.
\end{definition}

Now that we have all the ingredients to describe any non-generic situation, let us prove Proposition~\ref{prop:modelnganyg} by writing carefully the systems coming from non-genericity tuples and bounding their respective sizes (number of variables and degrees).

\begin{proof}[Proof of Proposition~\ref{prop:modelnganyg}]
    As in~\cite{AGS17}, we encode each possible non-generic situation by a normalized
    non-genericity tuple $(w,\mu,t,\varepsilon,M)$ in the sense of
    Definition~\ref{def:ngt}, and derive an associated polynomial system whose
    solution set corresponds to elements $D\in J[\alpha]$ such that: 
\begin{itemize}
  \item the reduced divisor $D$ of weight $w$ has the form $\sum_{i=1}^k\mu_i P_i$ with
    distinct $P_i$'s,
  \item each $\mu_i\,\alpha(P_i-P_{\infty})$ has weight $t_i$,
  \item each $\varepsilon_i$ is in $\{0,1\}$ and such that $\varepsilon_i=1$ if
    and only if $t_i=\mu_i=1$.
  \item the $k\times s$ matrix $M$ represents the points shared by the
    $\mu_i\,\alpha(P_i-P_{\infty})$ as in the discussion above, with
    $s\le gk$.
\end{itemize}

Following~\cite[Sec.~5.2]{AGS17}, let us write the equations associated to a
non-genericity tuple $(w,\mu,t,\epsilon,M)$.

First, we need variables for the coordinates of the $P_i$ such that the $\alpha$-torsion element is $D=\sum_{i=1}^k\mu_i(P_i-\infty)$, with $P_i\ne\pm P_j$ for all $i\ne j$. 
As a consequence, we introduce $2k$ variables for the coordinates $(x_i,y_i)$ of all the points $P_i$. Since these points are on the curve, they satisfy $y_i^2=f(x_i)$, however if $P_i$ is a ramification point this can be simplified into $y_i=0=f(x_i)$, which avoids multiplicities. We get a first set of equations
\begin{equation}\label{eq:eq1}\tag{Sys.1}    
  \left\{    
    \begin{array}{rl}    y_i^2=f(x_i)\not=0,\quad &\text{for all $i$ in $[1,k]$ such that    $\epsilon_i=0$}, \\    y_i=f(x_i)=0,\quad &\text{for all $i$ in $[1,k]$ such that         $\epsilon_i=1$}.\\    
    \end{array}\right.
  \end{equation}
  
  We model the fact that $P_i\ne \pm P_j$ for $i\ne j$ via the following set of inequalities:
  \begin{equation}\label{eq:eq2}\tag{Sys.2}  
    x_i\ne x_j,\quad \text{for all $i,j$ in $[1,k]$ such that $i\neq j$}. 
  \end{equation}
  
  The next step is to enforce the fact that the element $\mu_i\,\alpha(P_i-\infty)$ is of weight $t_i$. For the indices for which $t_i<g$, this is encoded by the equation defining $V_{\mu_i\,\alpha,t_i}$:

  \begin{equation}\label{eq:eq3}\tag{Sys.3}
    \left\{
    \begin{array}{l}
    \Delta_{\mu_i\,\alpha, t_i}(x_i) = 0,\\
    \Gamma_{\mu_i\,\alpha, t_i-1}(x_i) \ne 0,\\
    \end{array}
    \right.
        \quad \text{for all $i$
in $[1,k]$ such that $t_i<g$}.
\end{equation}
while for the indices for which $t_i=g$, this is encoded by the non-vanishing of the leading coefficient of the $\mu_i\,\alpha$-division polynomial:
\begin{equation}\label{eq:eq4}\tag{Sys.4}    
  d_{g}(x_i) \ne 0 ,\quad \text{for all $i$ in $[1,k]$ such that $t_i=g$}.
\end{equation}

We now need to model the fact that the $\mu_i\,\alpha(P_i-\infty)$ satisfy the conditions given by the matrix $M$. We write $\mu_i\,\alpha(P_i-\infty)=\langle u_i(X), v_i(X)\rangle$ in Mumford representation, where $u_i(X)$ and $v_i(X)$ correspond the $\mu_i\,\alpha$-division polynomials if $t_i=g$ or the non-generic division polynomials $\frak{u}_{\mu_i\,\alpha,t_i}$ and $\frak{v}_{\mu_i\,\alpha,t_i}$, if $t_i<g$. In both cases, these are polynomials in $X$ whose coefficients are polynomials in $x_i$ and $y_i$. Recall that the entries of $M$, denoted by$(m_{ij})_{i\in[1,k], j\in[1,s]}$, are such that $m_{ij}$ is the order of $Q_j$ in $\mu_i\,\alpha(P_i-\infty)$ if it is positive, or the opposite of the order of $Q_j'$ if it is negative. To this effect, we introduce $s$ new variables $\xi_j$ for the abscissae of the $Q_j$, and the following equations enforce the multiplicities:
\begin{align}    u_i^{(n)}(\xi_j)=0,\quad        & \text{for all $i,j$ in $[1,k]\times[1,s]$ and for all $n \le        \vert m_{ij}\vert-1$}\label{eq:eq5}\tag{Sys.5} \\    u_i^{(\vert m_{ij}\vert)}(\xi_j) \ne 0,\quad        & \text{for all $i,j$ in $[1,k]\times[1,s]$}\label{eq:eq6}\tag{Sys.6} \\    v_i(\xi_j)-v_{i'}(\xi_j)=0,\quad        & \text{for all $i,i',j$ such that $        m_{ij}m_{i'j}>0$}\label{eq:eq7}\tag{Sys.7} \\    v_i(\xi_j)+v_{i'}(\xi_j)=0,\quad        & \text{for all $i,i',j$ such that $        m_{ij}m_{i'j}<0$}\label{eq:eq8}\tag{Sys.8} \\    \xi_{j'} \not = \xi_{j},\quad &        \text{for all $j\not=j'$}.\label{eq:eq9}\tag{Sys.9}
\end{align} 
In Equations~\ref{eq:eq5} and~\ref{eq:eq6}, the notation $u_i^{(n)}$ is for the $n$-th derivative of $u_i$. This simple way of describing multiple roots is valid because the characteristic is large enough.

The next step of the construction is to consider a semi-reduced version of the divisor $\alpha(D) = \sum_{i=1}^k\mu_i\,\alpha(P_i-\infty)$. This semi-reduction process can be described directly on the matrix $M$: if two entries in a same column have opposite signs, a semi-reduction can occur (corresponding to subtracting the principal divisor of the function $(x-\xi_j)$), thus reducing the difference between these entries. This semi-reduction can continue until one of these two entries reaches zero. This whole process can be repeated as long as there are still columns containing entries with opposite signs.

Using this process, we compute a matrix $\widetilde{M}$ with the same dimensions such that if $M$ describes all the multiplicities in a divisor, then $\widetilde{M}$ describes all the multiplicities of a semi-reduced divisor equivalent to the input divisor. More precisely, the matrix $\widetilde M$ satisfies the following properties: (1) In each column, all elements are nonnegative; (2) The sum of the rows of $M$ equals the sum of the rows of $\widetilde M$; (3) For all $i,j$ such that $m_{i,j}$ is nonnegative, $\widetilde m_{ij}\leq m_{ij}$.

The function $\varphi$ that we will use to model the principality of the divisor $\alpha(D)$ will have two parts: a product of ``vertical lines'' corresponding to semi-reductions, and a part of the form $P(X)+YQ(X)$, where $P$ and $Q$ are coprime.  Modelling the existence of this second part requires to introduce new entities $\widetilde{u}_i$ that are the $u_i$ polynomials from which we remove the linear factors coming from semi-reduction as described by $\widetilde{M}$. Formally, we have the following equations defining $\widetilde{u}_i$:
\begin{equation}\label{eq:eq10}\tag{Sys.10}    
  u_i(X) = \widetilde{u}_i(X) \prod_{j=1}^s (X-\xi_j)^{\vert m_{ij}    
  \vert - \widetilde{m}_{ij}},\quad    \text{for all $i\in[1,k]$}.
\end{equation}
Indeed, by definition of the matrix $M$, the factor $(X-\xi_j)^{\vert m_{ij}\vert}$ divides exactly $u_i(X)$, and the factor $(X-\xi_j)^{\widetilde{m}_{ij}}$ divides exactly $\widetilde{u}_i(X)$. In order to express these conditions efficiently in the polynomial system, we introduce new variables for the coefficients of the $\widetilde{u}_i$ polynomials. Since we are now dealing with a semi-reduced divisor, we can consider its Mumford representation, i.e. two polynomials $U$ and $V$ with the following properties:
\begin{align}   
  & U = \prod_{i=1}^k \widetilde{u}_i, \quad U | V^2-f,    
  \label{eq:eq11}\tag{Sys.11} \\   
  & V \equiv v_i \bmod \widetilde{u}_i, \quad        
  \text{for all $i\in[1,k]$}. \label{eq:eq12}\tag{Sys.12}
\end{align}
The expression of $U$ is simple enough, so we do not have to introduce new variables for its coefficients. However, this will be necessary for the coefficients of the $V$ polynomial. Finally, in order to impose that the semi-reduced part of $\varphi$ has exactly the zeros described by this divisor, we have the equation
\begin{equation}\label{eq:eq13}\tag{Sys.13}   
  P + QV \equiv 0 \bmod U,
\end{equation}
which is expressed with new variables for the coefficients of $P$ and $Q$.

In Table~\ref{tab:variables}, we summarize all the variables used in the
polynomial system and count them. A key quantity for this count is the degree
of $U$ which is the sum of the degrees of the $\widetilde{u}_i$'s. It can be
computed directly from the tuple $(w,\mu,t,\epsilon,M)$. Then, to ensure
existence and unicity of the $V$ polynomial to represent the semi-reduced
divisor, we have to impose that $\deg V < \deg U$, so that we have exactly
$\deg U$ variables for the coefficients of $V$. For the polynomials $P$ and
$Q$, we need the degree of $P^2-Q^2f$ to be exactly $\deg U$. After a
normalization depending on the parity of $\deg U$, we get $\deg U-g$ variables
for their coefficients.

In the above process of turning the systems describing $J[\ell]$ into systems describing
$J[\alpha]$, we did not add any new variable, so that the study
of~\cite[Sec.~5.2]{AGS17} recalled in Table~\ref{tab:variables} is still valid
and in particular the total number of variables is bounded by $4g^2+g$.

\begin{table}[ht]    \centerline{%
    \begin{tabular}{lll}        \hline        Variables & Number of variables & Bound\\        \hline        Coordinates $(x_i, y_i)$ of $P_i$  & $2k$ & $2g$ \\        Abscissae $\xi_j$ of shared points & $s$, column-size of the matrix $M$ & $g^2$\\         Coefficients of the $\widetilde{u}_i$ polynomials & $\deg U = \sum_i (t_i        - \sum_j (\vert m_{ij} \vert - \widetilde{m}_{ij}) )$ & $g^2$\\        Coefficients of the $V$ polynomial & $\deg U$ & $g^2$ \\        Coefficients of the $P$ and $Q$ polynomials & $\deg U - g$ &        $g^2-g$ \\        \hline        Total   & $s+2k+3\deg U-g$ & $4g^2+g$ \\        \hline    \end{tabular}}    \caption{Summary of the variables in the polynomial system    corresponding to a normalized non-genericity tuple    $(w,\mu,t,\epsilon,M)$.}    
    \label{tab:variables}
  \end{table}
\begin{table}[ht]    \centerline{%
    \begin{tabular}{llll}        \hline        Equations reference & Number of equations (and bound) & $\deg_1$ & $\deg_2$ \\        \hline        Eq. and Ineq. \ref{eq:eq1} & $2k\le 2g$  & $2g+1$ & $\le 2$ \\InEq. \ref{eq:eq2} & $k(k-1)/2\le g(g-1)/2$ & $1$ & $0$ \\        Eq. and Ineq. \ref{eq:eq3}   & $\le 2g$ & $O_\eta(\ell^{3/g})$ & $0$ \\        InEq. \ref{eq:eq4} & $\le g$ & $O_\eta(\ell^{3/g})$ & $0$ \\        Eq. \ref{eq:eq5}   & $\sum_{i=1}^k \sum_{j=1}^s |m_{ij}|\le g^4$ &                    $O_\eta(\ell^{3/g})$ & $\le g$ \\        InEq. \ref{eq:eq6} & $ks\le g^3$ & $O_\eta(\ell^{3/g})$ & $\le g$ \\        Eq. \ref{eq:eq7} and \ref{eq:eq8} & $\le k^2s\le g^4$ & $O_\eta(\ell^{3/g})$ & $\le g$\\        InEq. \ref{eq:eq9} & $\le s^2\le g^4$ & $0$ & $1$ \\        Eq. \ref{eq:eq10}  & $\sum_{i=1}^k t_i\le g^2$ & $O_\eta(\ell^{3/g})$ & $\le g$\\        Eq. \ref{eq:eq11}  & $\deg U \le g^2$ & $0$ & $O(g^3)$ \\        Eq. \ref{eq:eq12} & $\sum_{i=1}^k \deg \widetilde{u}_i \le g^2$ &                    $O_\eta(\ell^{3/g})$ & $O(g^2)$\\        Eq. \ref{eq:eq13} & $\deg U\le g^2$ & $0$ & $O(g^3)$ \\        \hline    \end{tabular}}    \caption{Summary of the degrees of the equations in the polynomial    system corresponding to a normalized non-genericity tuple    $(w,\mu,t,\epsilon,M)$.}    \label{tab:degrees}
  \end{table}
As for the number of equations and their respective degrees, the only
difference with~\cite{AGS17} comes from the fact that the coefficients of the
$u_i$ and $v_i$ have degrees in the $x_i$'s bounded by $O_\eta(\ell^{3/g})$
instead of $O_\eta(\ell^3)$. For convenience, we also define $\deg_1$ as the
degree with respect to the variables $x_i$ and $\deg_2$ for all the other
indeterminates (we moved the variables $y_i$ to the second group because they
only appear with degree $\le 2$).

An updated version~\cite[Tab.~2]{AGS17} is given
by Table~\ref{tab:degrees}.  In particular, there are at most $O(g^4)$
equations involving at most $O(g^2)$ variables, and apart from the $x_i$'s, the
variables have degrees bounded by $O(g^3)$. This shows that any system
corresponding to a non-genericity tuple satisfies the degree conditions of
Proposition~\ref{prop:modelnganyg}. As in~\cite{AGS17}, the number of such
tuples is bounded by $g^{O(g^3)}$ and Proposition~\ref{prop:modelnganyg} is
proved. 
\end{proof}

\section{Complexity analysis}\label{sec:cplxrmangyg}

Now that we have modelled subsets of $J[\alpha]$ by polynomial systems whose
sizes in terms of equations, variables and degrees have been carefully bounded,
we apply the geometric resolution algorithm and bound its complexity.

\subsection{Solving the polynomial systems modelling \texorpdfstring{$J[\alpha ]
$}{}}\label{sec:cplxsolsys}

Just as in~\cite{AGS17}, we use geometric resolutions to describe $0$-dimensional (i.e.
finite) sets $V\subset \overline{\mathbb F_q}^n$ where $V$ is defined
over $\F_q$. The terminology
here is borrowed from \cite{cafure2006fast}, see also \cite{GiuLecSal01}. 

\begin{definition}[Geometric resolution]
An  $\mathbb F_{q^e}$-geometric resolution of $V$ is a tuple 
$((\ell_1,\ldots, \ell_n), Q, (Q_1,\ldots, Q_n))$ where:
\begin{itemize}
  \item The vector $(\ell_1,\ldots,\ell_n)\in\mathbb F_{q^e}^n$ is such
    that the linear form 
    $$\begin{array}{rccc}\ell : &\overline{\mathbb
      F_q}^n&\rightarrow&\overline{\mathbb F_q}\\
      &(x_1,\ldots, x_n)&\mapsto &\sum_{i=1}^n
      \ell_i x_i\end{array}$$ takes distinct values at all points in $V$. The linear form
    $\ell$ is called the primitive element of the geometric resolution;
  \item The polynomial $Q\in\mathbb F_{q^e}[T]$ equals
    $\prod_{\mathbf x\in V}(T-\ell(\mathbf x));$
  \item The polynomials $Q_1,\ldots, Q_n\in\mathbb F_{q^e}[T]$ parametrize $V$
    by the roots of the polynomial $Q$, i.e.
    $$V = \{(Q_1(t),\ldots, Q_n(t))\mid t\in\overline{\mathbb F_q},  Q(t) = 0\}.$$
\end{itemize}
\end{definition}
We will need to bound the complexity of computing geometric resolutions of
bihomogeneous polynomial systems. We do so by using a variant
of~\cite[Prop.~3]{AGS17}, which is restated here.

\begin{proposition}\label{prop:prop3}\cite[Prop.~3]{AGS17}
  There exists a probabilistic Turing machine $\mathbf T$ which takes as input
polynomial systems with coefficients in a finite field $\mathbb F_q$
  and which satisfies the following
  property. For any function $h:\Z_{>0}\rightarrow \Z_{>0}$, for any
  positive number $C>0$ and for any $\varepsilon>0$, there exists a function
  $\nu:\Z_{>0}\rightarrow\Z_{>0}$ and a positive number
  $D>0$ such that for all positive integers $g, \ell, n_x, n_y, d_x, d_y, m >
  0$ such that $n_x < C\, g$, $n_y<h(g)$, $d_x<h(g)\, \ell^C$, $d_y<h(g)$,
  $m<h(g)$, for any prime power $q$ such that the prime number $p$ dividing $q$
  satisfies $2^{n_x+n_y}d_x^{n_x}\, d_y^{n_y} < p$, and for any polynomial system $f_1,\ldots, f_m\in
  \mathbb F_q[X_1,\ldots, X_{n_x}, Y_1,\ldots, Y_{n_y}]$ such that 
  \begin{itemize}
    \item for all $i\in
      [1,m]$, $\deg_x(f_i)\leq d_x$ and $\deg_y(f_i)\leq d_y$,
    \item the ideal $I= \langle f_1,\ldots, f_m\rangle$ has dimension $0$ and is
  radical,
  \end{itemize}
  the Turing machine $\mathbf T$ with input $f_1,\ldots, f_m$ returns an
    $\mathbb F_{q^{\lceil\nu(g)\log\ell\rceil}}$-geometric
      resolution of the variety $\{\mathbf x\in\overline{\mathbb F_q}\mid
      f_1(\mathbf
    x) = \dots = f_m(\mathbf x) = 0\}$ with probability at least $5/6$, using
      space and time bounded above by $\nu(g)\,
      \ell^{D\,g}\,(\log q)^{2+\varepsilon} $.
\end{proposition}

\begin{proof}
  This is done in~\cite[Sec.~3]{AGS17}.
\end{proof}

\begin{proposition}\label{prop:sysanyg} 
  For any $\varepsilon>0$, there is a constant $D$ such that for any
  endomorphism $\alpha\in\Z[\eta]$ of norm a multiple of $\ell >g$ coprime to the
  base field characteristic, there is a Monte Carlo algorithm which computes an
  $\mathbb F_{q^e}$-geometric resolution of the sub-variety of $J[\alpha]$
  consisting of $\alpha$-generic $\alpha$-torsion elements, where
  $e=O_\eta(\log\ell)$.  The time and space complexities of this algorithm are
  bounded by $O_\eta(\ell^{D}(\log q)^{2})$ and it returns the correct
  result with probability at least $5/6$.
\end{proposition}
\begin{proof}  
  Let us consider the sub-variety $S\subset J[\alpha]$ consisting of
  $\alpha$-generic elements, and $I$ the corresponding ideal. More precisely, we
  see $I$ as the ideal of a sub-scheme of the scheme $J[\alpha]$, itself
  subscheme of $ J[\deg\alpha]$, which is the kernel of a finite and étale map
  because $\deg\alpha$ is a small multiple of $\ell$ and is hence coprime to the
  characteristic $p$ thanks to our assumptions on the size of $p$ in the
  statement of Theorem~\ref{th:mainanyg}.

  Therefore, $I$ is 0-dimensional and radical. Since all the elements in
  $S$ have the same weight~$g$ we can use the Mumford coordinates
  $\langle u(X), v(X)\rangle$ with $\deg u=g$ and $\deg v<g-1$ as a
  local system of coordinates to represent them. But the polynomial
  system that we have built is with the $(x_i,y_i)$ coordinates, that
  is, it generates the ideal
  $I^\mathrm{unsym}$ obtained by adjoining to the equations defining
  $I$ the $2g$ equations coming from $u(X) = \prod(X-x_i)$ and
  $y_i=v(x_i)$. Then we have $\deg I^\mathrm{unsym} = g! \deg I$.
  By the $\alpha$-genericity condition, all the fibers in the variety
  have exactly $g!$ distinct points corresponding to permuting the
  $(x_i,y_i)$ which are all distinct. Therefore the radicality of $I$
  implies the radicality of $I^\mathrm{unsym}$ and we can apply the modified
  version of~\cite[Prop.~3]{AGS17} to our polynomial system.

  These systems are very similar to those presented in~\cite{AGS17}, which is
  the reason why we will be using Proposition~\ref{prop:prop3}. In this paper,
  however, we bound $d_x$ by some $h(g)\, \ell^{3/g}$ instead of $h(g)\,
  \ell^3$. Following the proof provided in~\cite[Sec.~3]{AGS17}, the factor
  $1/g$ in the exponent propagates which yields a final complexity bound
  bounded by $\nu(g)\, \ell^{D}\,(\log q)^{2+\varepsilon} $ (the exponent of
  $\ell$ is now a constant). 
  
  Indeed, by Proposition~\ref{prop:modelanyg} we now have a function $h$ such
  that $d_x\le h(g)\ell^{3/g}$ instead of $h(g)\ell^3$. As we remarked, we can
  propagate this factor $1/g$ and compute an $\F_{q^e}$-geometric resolution of
  $S$ in time and space bounded by $O_\eta(\ell^{D}(\log q)^{2+\varepsilon})$,
  with $e=O_\eta(\log\ell)$, using the result of Proposition~\ref{prop:prop3}
  with $C=3$. Note that by our definition of $O_\eta()$ the $\varepsilon$ can
  be removed.
  \end{proof} 

  \paragraph*{Remark.}
  The bottleneck of this algorithm is the computation of geometric resolutions
  of polynomial systems which is quadratic in $\delta$ the maximum of the
  degrees of the intermediate ideals $\langle f_1,\ldots,f_i\rangle$ (see for
  instance~\cite{GiuLecSal01} for a detailed complexity analysis). This
  $\delta$ is hard to assess, but it is bounded by the (multihomogeneous)
  B\'ezout bound, and we bound it by $2^{g+n_y} d_x^{g} d_y^{n_y}$
  using~\cite[Prop.~8]{AGS17} (itself derived from~\cite[Prop.
  I.1]{SafSch17}). Neglecting factors in $O_\eta(1)$, $\delta$ is in
  $O_\eta(d_x^{g})$. The exponent $D$ is essentially determined by $\delta$,
  more details we be given when explicitly computing $D$ in the next section. 

  Following the same proof but invoking Proposition~\ref{prop:modelnganyg}
  instead of Proposition~\ref{prop:modelanyg}, the same complexity bound holds
  for solving the polynomial system associated to any non-genericity tuple. Even
  if a non-zero $\alpha$-torsion element is only found after solving all the
  systems associated to non-genericity tuples, the cost for computing
  $\psi\bmod\ell$ is only multiplied by a factor in $O_\eta(1)$.

  \subsection{An explicit bound for the exponent of $\log q$}

  From the result of Proposition~\ref{prop:sysanyg}, we can compute the
  elements $D_i$ of Algorithm~\ref{algo:rmanyg} from which we deduce
  $\psi\bmod\ell$ in $O_\eta(\ell^{D}(\log q)^{2+\varepsilon})$ bit operations.
  However, the use of the geometric resolution algorithm makes this a
  Monte-Carlo algorithm while we claim that our point-counting algorithm is a
  Las Vegas one. This easily fixed because once an element $D_i$ is computed
  using this Monte-Carlo algorithm, we can check for a negligible cost that
  this $D_i$ has the required property (it is a non-zero element of order
  $\ell$ in $J[\alpha_i]$). Then if it turns out that our Monte-Carlo algorithm
  did not return a correct output, we simply repeat until it succeeds. Since
  the probability of success is lower-bounded by a positive constant, the expected
  runtime of the resulting Las Vegas algorithm is the runtime of the
  Monte-Carlo algorithm up to multiplication by a constant.

  We have proven that there exists a constant $D$ such that for any prime
  $\ell$ satisfying conditions (C1) to (C4), computing $\psi\bmod\ell$ is
  achieved within $O_\eta(\ell^{D}(\log q)^{2+\varepsilon})$ bit operations.
  Since both the number of such primes $\ell$ and the size of the largest prime to
  consider are in $O_\eta(\log q)$, the overall complexity of our
  point-counting algorithm is in $O_\eta( (\log q)^{D+3})$.

  Now it only remains to compute an explicit value for $D$, which we do by
  following the proof of~\cite[Prop.~3]{AGS17}. Going straight to the point,
  the dominant part in the complexity analysis that is done in the proof is in
  $O_\eta(d_x\delta^2\log q+\delta^2(\log q)^2)$, where $\delta$ is as in the
  previous remark. From the degree bound of Prop~\ref{prop:modelanyg}, $\delta$
  is in $O_\eta(\ell^3)$ and so the complexity of solving the systems is in
  $O_\eta(\ell^{6+3/g}\log q+\ell^6(\log q)^2)$. Since we have better bounds
  for point-couting in genus $\le 3$, we can assume that $g>3$ and since
  $\ell=O_\eta(\log q)$, the second term of the sum is the dominant one and so
  the $D$ of Proposition~\ref{prop:sysanyg} can be chosen equal to 6. From the
  previous paragraph, it follows that our point-counting algorithm runs in time
  $O_\eta((\log q)^9)$.
  
  Note that our bound on $d_x$ is pessimistic because we used the proven cubic
  bound for the degrees of Cantor's division polynomials while we expect them
  to be actually quadratic (see the final remark of~\cite[Sec.~6]{AGS17} for
  detailed experiments and conjectures). This bound was achieved thanks to
  recurrence formulas for Cantor's polynomials that are provided in~\cite{Ca94}
  but it does not seem possible to do better than a cubic bound using them. To
  prove the quadratic bounds in genus 3, another set of formulas also given
  in~\cite{Ca94} were used. However, they have a bad dependency in the genus
  $g$ and give a bound that is worse than cubic for $g\ge 5$, which is the
  reason we do not use them here. 
  
  Assuming that we can prove a quadratic bound for the degrees of Cantor's
  polynomials, $d_x$ is reduced to $O_\eta(\ell^{2/g})$ so that $\delta$ is in
  $O_\eta(\ell^2)$ and so $D$ is bounded by 4 instead of 6. Thus, the overall
  complexity would therefore be in $O_\eta(\log^7 q)$ for any $g$. 
  
  Since we have removed the dependency in $g$ from the exponent of $\log q$, it
  is natural to investigate further how the factor hidden in the $O_\eta()$
  notation grows when $g$ grows. This is what we do in the next section.

\subsection{Dependency in $g$ of the complexity}

The goal of this section is to assess the potential of our algorithm to achieve
a polynomial-time complexity both in $g$ and $\log q$ on some family of curves.
To this end, we review our complexity analysis with additional attention given
to the factors that previously vanished in the $O_\eta$.

\paragraph{Dependency in $g$ of the largest $\ell$.}

Let us first come back to the constant $C_\eta$ of Section~\ref{sec:psipi}. We have
seen that the only non-polynomial dependency in $g$ came from the matrix norm
when inverting the linear change of variables
$\psi_k=\sum_{i=0}^{g-1}a_i\eta_k^i$, which is described by the Vandermonde
matrix of the $g$ conjugates of $\eta$, denoted by $\eta_k$ for
$k\in\{1,\ldots,g\}$. Let $B$ be the inverse of this matrix, then we have
\[B_{ij}=\frac{\displaystyle\sum_{\substack{1\le k_1 <\cdots < k_{g-j} < g\\ k_1,\ldots,
k_{g-j}\ne i}} (-1)^{j-1}\eta_{k_1}\cdots\eta_{k_{g-j}}}{\eta_i\displaystyle\prod_{k\ne
i}(\eta_k-\eta_i)}.\] Let $E=\max_k(|\eta_1|,\ldots,|\eta_k|)$,
$e=1/\min_k(|\eta_1|,\ldots,|\eta_k|)$, and $D=\max_{i\ne j}\left(\vert
\eta_i-\eta_j\vert^{-1}\right)$, then we can bound the absolute value of any
entry of $B$ very roughly either by $ge(2ED)^g$ or by $ge$ if $2ED\le 1$, and
the matrix-norm of $B$ is bounded by $g$ times this previous bound. Note that
the possible denominators on the $a_i$ are also a nuisance but they are bounded
by the discriminant of $\Z[\eta]$.  This discriminant is in turn bounded by
$\max_{i\ne j}\left(\vert \eta_i-\eta_j\vert\right)^{2g}$.  Thus, the constant
$C_\eta$ can be bounded by $g^2c^g$, where $c$ has a polynomial dependency in
$\eta$ and its conjugates. 

By the prime number theorem, the set $L$ of primes such that $\prod_{\ell\in
L}\ell > 2C_\eta\sqrt{q}$ is such that the number and size of primes in $L$ is in
$\softO(g\log q)$. As we already mentioned, the primes to consider
must satisfy the conditions $(C1)$ to $(C4)$ and that may cause them to be
larger by a factor depending exponentially on $g$ {\it a priori}. Since the
complexity of computing $\chi_\pi\bmod\ell$ is polynomial in $\ell$, this
implies that the overall complexity depends exponentially in $g$ in general. 

However, a curve in the family $\C_{n,t}$ introduced in
Section~\ref{sec:rmfamily} has RM by the real subfield of $\Q(\zeta_n)$, for
which we know that the proportion of split primes is $2/(n-1)=1/g$. Therefore,
this first obstacle due to the size of primes to consider can be overcome
provided that we further strengthen the assumptions on the RM-curves we
consider.

\paragraph{Finding small elements in lattices.}
This time, the exhaustive search is no longer sufficient for our needs because
of the exponential factor $2^g$ in the size of the ball $\left\lbrace v\mid
||v||_{\infty}\le\left[\mathcal{O}_{\Q(\eta)}:\Z[\eta]\right]^{-1/g}\ell^{1/g}\right\rbrace$.
Unfortunately, the currently known algorithms for finding short vectors in time
subexponential in the dimension of the lattice have a drawback that makes them
unusable in our point-counting algorithm. Indeed, although they run faster than
the naive approach, they do not necessarily output the shortest non-zero vector
on the lattice, but an approximation that may be greater by a factor which is
also subexponential in the dimension. The size of the short vector plays a
prominent role in the complexity analysis of our point-counting algorithm as it
gives a bound on the degrees of the equations modelling $J[\alpha]$. Even if we
find an $\alpha$ whose coordinates are in $C\ell^{1/g}$, the constant factor
$C$ will cause a factor $C^g$ in the bound $2^{g+n_y} d_x^{g} d_y^{n_y}$, and
hence in the final complexity of solving the polynomial systems.

Although finding short generators of ideal in number fields is believed to be
hard in general, we may still expect to further restrict the RM curves we
consider so as to fall in a case for which the complexity of such task becomes
affordable. Examples are given in~\cite{bauch17}, where a classical algorithm
is shown to compute short generators of principal ideals in particular number
fields called multiquadratics, i.e. fields of the form
$\Q(\sqrt{d_1},\ldots,\sqrt{d_n})$, in time quasipolynomial in the degree
(which is $g$ in our context).  While we acknowledge that it is quite
speculative to hope for families of curves of arbitrary high genus with RM by a
$\Z[\eta]$ satisfying all the previous hypotheses, we do not linger on this
because the next point is much more of a concern anyway.

\paragraph{Solving polynomial systems.}

Using the strategy of Section~\ref{sec:cplxrmanyg}, the complexity is quadratic
in the bound $2^{g+n_y} d_x^{g} d_y^{n_y}$ of~\cite[Prop.~8]{AGS17}, which
includes a factor $g^{O(g^2)}$. Indeed, although the ideals of $\alpha$-torsion
have degree $\ell^2$ independent of $g$, this is not true for the number of
variables involved in our modelling, which is at least $g^2$ in the generic
case.  

However, even if none of the current complexity bounds for solving polynomial
systems is sufficient to derive a polynomial-time algorithm both in $g$ and
$\log q$, there are still reasons to hope. Indeed, while the analysis made
in~\cite{abelard} pointed out the fact that the systems themselves could have
exponential size in $g$, these fears were based on very rough estimates of
their size as straight-line programs. In fact, the cost of evaluating our
equations of the form $P+Qv_i= 0 \bmod u_i$ can be split into two parts : first
computing $u_i$ and $v_i$, which amounts to computing $\alpha( (x_i,y_i)
-P_\infty)$ in $\F_q[x_i,y_i]/(y_i^2-f(x_i))$. This is done within
$O(||\alpha||_\infty/g\log\ell+g^2)$ operations on polynomials whose sizes are
bounded by $O(g\||\alpha||_\infty \ell^{3/g})$ field elements. Then, one has to
finally reduce the degree-$g^2$ polynomial $P+Qv_i$ modulo the degree-$g$
polynomial $u_i$, which can be done naively by replacing powers of $X$ larger
than $g$, for at most $g^4$ operations on polynomials of degrees $\le g^2$ with
coefficients in $\Frac\left(\F_q[x_i,y_i]/(y_i^2-f(x_i))\right)$ whose sizes
are bounded by $2g^3\ell^{3/g}$ field elements. 

Thus, our systems have polynomial sizes in both $g$ and $\log q$, which still
fosters the hope that it could still be possible to solve them in time also
polynomial in these parameters, although we recognize that improving on the
estimate given by the multihomogeneous B\'ezout bound would be a significant
progress. Other possible workarounds to avoid an exponential dependency in $g$
could be looking for easier instances in which we could model the
$\alpha$-torsion by even smaller polynomial systems, or cases for which there
are simpler ways of obtaining a generic $\alpha$-torsion divisor than the one
we used.

\section{Future work}

Based on the facts that the genus-3 RM point-counting algorithm of~\cite{AGS18}
is practical and that we extended it to arbitrary genus with a similar
complexity (at least conjecturally), one could hope to use it for practical
computations in genus larger than 3. In the current state, the exponential
dependency in $g$ and the difficulties that were already encountered in genus 3
make it unrealistic, and we also lack an open and competitive implementation of
the geometric resolution algorithm.

Proving the quadratic bound on the degrees of Cantor's polynomials still has to
be done in order to prove that we have a complexity result close to the
genus-3 case. Cantor's original paper is quite long and technical but also
provides recurrence formulas that remains relatively simple. However, a
straightforward use of these formulas is not sufficient to establish tight
bounds in genus larger than 3. Maybe a deeper and more technical analysis of
intermediate results presented in Cantor's paper~\cite{Ca94} could yield
sharper bounds but we leave this subject to further research.

An interesting problem that could have both practical and theoretical impact in
terms of complexity is to find new (families of) curves with explicit real
multiplication. While RM by multiquadratics is theoretically interesting to
control the exponential dependency in $g$, finding curves with RM by an order in
which a small prime (say $\ell\le 11$) happens to be totally split could be a
first step towards practical experiments in genus $\ge 4$.

Lastly, even if we were to find a way of solving the polynomial systems within
a polynomial (or at least subexponential) complexity, the number of non-generic
systems is still exponential in $g$. Heuristically, non-genericity should never
be a problem, but in order to reach a proven subexponential complexity, one
also needs to find another way of dealing with non-genericity.

\paragraph{Acknowledgements.}
Most of this work already appears as Chapter VII in the author's thesis
manuscript~\cite{abelard}. As such, the author received helpful feedback from
his advisors Pierrick Gaudry and Pierre-Jean Spaenlehauer; and from his thesis
referees Christophe Ritzenthaler and Fréderik Vercauteren.  The author is also
grateful to Benjamin Smith and David Kohel for pointing out references and for
fruitful discussions. The author is indebted to the anonymous reviewers for numerous
improvements to the clarity of the paper as well as for pointing out an error
in the exponent of $\log q$. 

\bibliographystyle{plain}
\bibliography{biblio}
\end{document}